\documentclass[11pt]{amsart}
\usepackage{amsfonts,amssymb,amsmath}
\usepackage{amsmath}
\usepackage{color}

\newtheorem{theorem}{\bf Theorem}[section]

\newtheorem{remark}{\bf Remark}[section]
\newtheorem{lemma}{\bf Lemma}[section]
\newtheorem{definitions}{\bf Definition}[section]
\numberwithin{equation}{section}

\title[Risk-sensitive control of reflected diffusion Process]
{Risk-sensitive control of reflected diffusion Processes on Orthrant}

{\author{Sunil Kumar Gauttam,  K. Suresh Kumar and Chandan Pal }
\address{Department of
Mathematics, The LNM Institute of Information Technology , Jaipur-302031
, Tel no. +91 141 519 1800
, Fax no. +91 141 518 9214, India.}
\address{Department of
Mathematics, Indian Institute of Technology Bombay, Mumbai -
400076, Tel no. +91 22 2576 7489, Fax no. +91 22 2572 3480, India. }
\address{ Department of Mathematics, Indian Institute of Science, Bangalore-560012, Tel no. +91 80 2293 3210, Fax no. +91 80 2360 0146, India. }

\email{sgauttam@lnmiit.ac.in, suresh@math.iitb.ac.in, chandan14@math.iisc.ernet.in }
\begin{document}
\maketitle

\begin{abstract}
\noindent In this article, we prove the existence of optimal risk-sensitive control with state
constraints. 
We use near monotone assumption  on the running cost to
prove the existence of optimal  risk-sensitive control.
\end{abstract}

\vspace{10mm}


\vspace{5mm}

\section{Introduction and Problem Description}\label{PD}
\setcounter{equation}{0}
In this paper we study the risk-sensitive control problem when the state dynamics 
is governed by a controlled reflecting stochastic differential equation in $d$-dimentional orthant. 
We prove that the risk-sensitive value is an eigenvalue of the nonlinear eigenvalue problem with oblique
boundary conditions (see, the equation 
(\ref{ergodic}) ) which is the so called Hamilton Jacobi Bellman (HJB) equation of the risk-sensitive 
control problem with state constraints. We also show that any minimizing selector in 
(\ref{ergodic})   
corresponding to the eigen function of the risk-sensitive value is a risk-sensitive optimal control.  
We use near monotone 
structural condition on the running cost and a blanket recurrence condition for the state dynamics
for proving this result.

The paper is organized as follows. The remaining part of Section 1 contains the 
detailed description of the problem and some results on controlled reflected stochastic differential
equations which are used in  subsequent sections. In Section 2, we discuss an auxillary risk-sensitive
control problem with discounted cost structure. We prove the existence of optimal value and 
control without the structural condition near monotonicity on the running cost. In the  final section,
we prove our main theorem, i.e. Theorem 3.2. The proof is based on the socalled vanishing 
discounting method. 

\vspace{.1in}

Let $U$ be a compact metric space and $D$ denote the positive orthrant of $\mathbb R^d$, i.e., 
$$D :=\{ x\in \mathbb R^d: x_i>0,\ \forall \ i=1,2,\cdots,d\}.$$ 
 Let $\overline{A}, \, \partial A$ denote the closure  and boundary of the set $A$, for any subset $A$ of $\mathbb{R}^d$ respectively. 
\paragraph{}  For the given functions $ b : \overline{D} \times U \longrightarrow \mathbb R^d,$ $
\sigma : \overline{D} \longrightarrow \mathbb R^{d\times d}$ and $\gamma : \mathbb R^d \longrightarrow \mathbb R^d $, consider the   controlled 
reflected diffusion in  $\overline { D}$, given by the solution of 
the reflected stochastic differential equation (in short RSDE)
\begin{eqnarray} \label{state}
\begin{array}{rcl}
\displaystyle{
 dX_t} &=& \displaystyle{ b(X_t,v_t)dt+\sigma(X_t) dW_t - \gamma(X_t)d\xi_t }, \\
\displaystyle{ d\xi_t }&=&  \displaystyle{  I_{\{X_t\in \partial {D} \}}d\xi_t  }, \\
\displaystyle{ \xi_0 } &=& \displaystyle{  0,\ \ \ \ X_0 = x\in \overline{D}, } 
\end{array}
\end{eqnarray}
where $W =(W_1,\cdots, W_d)$ is an $\mathbb R^d$-valued standard Wiener 
process, $v(\cdot)$ is a $U$-valued measurable process  non anticipative with respect to $W(\cdot)$,
called an admissible control. In fact the pair $(v(\cdot), W(\cdot))$ defined on a filtered probability 
space $(\Omega, \mathcal{F}, \{ \mathcal{F}_t \} , P)$ satisfying the usual hypothesis is an
 admissible control if and only if  $v(\cdot)$ is measurable and ${\{ \mathcal F}_t \}$-adapted, 
see Remark 2.1, p.31 of \cite{arapostathis_borkar_ghosh}.  Henceforth, all
filtered probability spaces are assumed to satisfy usual hypothesis.
The set of all admissible control is denoted by $\mathcal{A}$. 

 By a solution to (\ref{state})  we mean a pair of continuous time   processes 
 $(X(\cdot),\xi(\cdot))$ satisfying (\ref{state}) such that the process $X(\cdot)$ is 
$\overline { D}$-valued and  $\xi(\cdot)$ is 
a non-decreasing process which increases only when $X(\cdot)$ hits the boundary $\partial {D}$. 
The above is a special case of the more general definition of solutions of SDEs with reflection, 
see \cite{dupuis_ishii}. In fact we consider the  case when the direction of reflection is single valued.

We use the relaxed control frame work given as follows. The compact metric space $U = {\mathcal P}(S)$
for some compact metric space $S$, where ${\mathcal P}(S)$ denote the space of probability measures
on $S$ endowed with the Prohorov topology, i.e. the topology induced by weak convergence. The drift
coefficient $b$ takes the form 
\[
b(x, v) \ = \ \int_{S} \Bar{b}(x, s) v(ds) , \ v \in U, x \in \overline{D}.
\]
 \paragraph{} 
For $l = 1, 2, \cdots $, set
\[
D'_l \ = \ D \cap B(0, l), \ B(0, l) \ = \ \{ x \in \mathbb{R}^d| \|x \| < l \}.
\]
From the proof of Theorem A2 (ii) and the remark in  p. 28 of \cite{Chung} 
there exists open domains $D_{lm} \subseteq \mathbb{R}^d$  with $C^\infty$ boundary such that 
\begin{itemize}
\item The distance between $\partial D'_l$ and $D_{lm}$ satisfies, 
$$d ( D_{lm} , \partial D'_l) < \frac{1}{m}, \ l \geq 1,$$
\item $D_{ln} \subseteq D_{lm}, \ n \geq m , \ l \geq 1$.
\end{itemize}
Set  $$ D_m \ = \ \cup^\infty_{l =1} D_{lm},  \ m \geq 1.$$ Then we have 

\begin{itemize}
 \item[(i)]  For each $m \geq 1$, $D_m$ is with $C^\infty$ smooth boundary and 
$D_m \downarrow \Bar{D}.$
\item[(ii)] For any compact set $C\subset \Bar{D}$, we have $C\subset \overline{D}_{lm}$ for $m \geq 1$ 
and $l$ sufficiently large. 
\end{itemize}

We make the following assumption which is sufficient to  ensure the existence of unique
 solution to the equation (\ref{state})\\
{\bf (A1)}\ (i) The function $\Bar{b}$ is bounded continuous, Lipschitz continuous in its first
argument uniformly with respect to the second argument. \\
\ \ \ (ii) The functions $\sigma_{ij} \in C^2 (\Bar{D}), i, j = 1, \cdots , d $  and bounded. \\
\ \ \ (iii) The function $a\stackrel{def}{=} \sigma\sigma^\bot$ is uniformly elliptic with ellipticity constant $\delta$, i.e.,
\[
x  a(x) x^{\bot} \, \geq \, \delta \, |x|^2, \ x \in \overline{D} \, ,
\]
where $x^{\bot}$ denote the transpose of the vector $x$.\\
{\bf (A2)} (i) The function $\gamma   = (\gamma_1, \cdots , \gamma_d) $  is such that 
$\gamma_i \in C^2_b(\mathbb{R}^d)$, and there exists $\eta>0$ such that 
\begin{eqnarray*}\label{no_slip}
\displaystyle { \gamma(x)\cdot n_m(x) } &\geq& \displaystyle { \eta \ \mbox{ for all } x\in  \partial D_m,}
 \end{eqnarray*}
here $ n_m(\cdot)$ denote the outward normal to $ \partial D_m$. \\
(ii) There exists a symmetric matrix valued map $ M: \mathbb{R}^d \longrightarrow \mathbb{R}^d
\otimes \mathbb{R}^d , \, \mathbb{R}^d \otimes \mathbb{R}^d$ the set of all $d \times d$
real valued matrices with usual metric, such that $M = (m_{ij})$, $m_{ij} \in C_b(\mathbb{R}^d)\cap W^{2,\infty}(\mathbb{R}^d)$ for $i,j=1,2,\cdots,d$ and satisfies the following\\
(a) there exists $\delta_1$ such that
$$ x^{\bot} M x \geq \delta_1 \|x\|^2, x\in \mathbb{R}^d ;$$
(b) there exists $C_0 >0$ such that
$$  C_0 \|x-y\|^2+\sum_{i,j}^d m_{ij}(x)(x_i-y_i)\gamma_j(x)\geq 0, \; \mbox{for  all}\; x\in \partial D, y\in \overline{D} ;$$
(c) Let $z \in \overline{D}$ and if for some $C_0>0$
$$  C_0 \|x-y\|^2+\sum_{i,j}^d m_{ij}(x)(x_i-y_i)z_j(x)\geq 0, \; \mbox{for  all}\; x\in \partial D, y\in \overline{D} ;$$
then $z=\theta \gamma(x)$ for some $\theta >0$.\\

The existence
of a unique  weak solution of (\ref{state}) for an admissible control has been proved in [\cite{ghosh_suresh},\cite{lions_sznitman}] using the following programme. First establish the existence of unique
strong solution with zero drift as follows.
\begin{itemize}
\item Establish the existence of a solution to (\ref{state}) in the smooth domain $\overline{D}_m,
\, m \geq 1$,
\item  use convergence arguments to obtain a solution of (\ref{state}) in $\overline{D}$,
\item establish pathwise uniqueness, see Lemma 3.3 of \cite{Suresh_Bagchi}.
\end{itemize} 
Now with non zero drift, using Girsanov transformation method to establish existence of unique weak
solution under admissible controls, see [\cite{arapostathis_borkar_ghosh}, pp-42-44]. 
For a Markov control, one can prove the existence of unique strong
solution by adapting the approach by Zovokin and Veretenikov, 
see [ \cite{arapostathis_borkar_ghosh}, pp.45-46] for the analogous proof for the unconstrainted 
diffusions.  See Theorem 3.2 of \cite{Suresh_Bagchi} for details.

The running cost function $ r:\overline{ D }\times U\longrightarrow [0,\infty)$ is given 
in the relaxed frame work as 
\[
r(x, v) \ = \ \int_S \Bar{r}(x, s) v(ds),  x \in \overline{D}, v \in U.
\]

 Throughout this paper we assume that the cost function $\Bar{r}$ is  continuous in $(x,s)$ and  Lipschitz continuous in the first argument uniformly with respect to the second. 
We consider two risk-sensitive cost criteria, discounted cost and ergodic cost criteria which is described
below.

\subsection{Discounted cost criterion} Let $\theta \in (0, \ \Theta)$ be the risk-aversion parameter. In the $\alpha$-discounted cost criterion, controller chooses his control $v(\cdot)$ from the set of all admissible controls ${\mathcal A}$ to minimize his $\alpha$-discounted risk-sensitive cost given by 
\begin{equation}\label{discountedcost}
J^{v}_{\alpha}(\theta, x) \ := \ \dfrac{1}{\theta} \ln  E^{v}_x
\Big[ e^{\theta \int^{\infty}_0 e^{-\alpha t} r(X_t, v_t) dt} \Big] , x \in \overline{D}, 
\end{equation}
where $\alpha > 0$ is the discount parameter,  $X(\cdot)$ is the solution of the s.d.e. (\ref{state}) corresponding to $v(\cdot) \in  \mathcal{A}$ and $E^{v}_x$ denote the expectation with respect
to the law of the process (\ref{state}) corresponding to the admissible control $v$ with the initial condition $X_0 = x$.  An admissible control $v^*(\cdot)\in \mathcal A$ is called 
optimal control if
$$
J^{v^*}_{\alpha}(\theta, x) \leq J^{v}_{\alpha}(\theta, x), \ \ \mbox{for all} \ v(\cdot)\in \mathcal A \; \mbox{and } \; x \in \overline{D}.
$$

\subsection{Ergodic cost criterion}
In this criterion controller chooses his control $v (\cdot) \in {\mathcal A}$ so as to minimize his risk-sensitive
accumulated cost given by
\begin{equation}\label{risksensitivecost}
\rho^{v} (\theta, x) \ = \limsup_{T \to \infty} \frac{1}{\theta T} \ln E^{v}_x
\Big[ e^{\theta \int^T_0 r(X_t, v_t) dt} \Big] , x \in \overline{D}.
\end{equation}
The definition of  optimal control is analogous. From now onwards, we take $\Theta = 1$ without any
loss of generality. 

\subsection{ Various subclasses of controls}An admissible control $v(\cdot)$ is said to be  a Markov control if there exists a measurable map $\bar{v}:[0,\infty)\times \overline{D} \longrightarrow U$ such that $v(t)=\bar{v}(t,X(t))$. By an abuse of notation, the measurable map  $\bar{v}:[0,\infty)\times \overline{D} \longrightarrow U$, itself is called Markov control. If $\bar{v}$ has no explicit time dependence then it is said to be a stationary Markov control. We denote the set of all Markov control and stationary Markov control by $\mathcal{M}$ and $\mathcal{S}$ respectively. An admissible control $v(\cdot)$ is said to be 
a feedback control if it is progressively measurable with respect to $\{{\mathcal F}^{X, \xi}_t \}$, where $(X(\cdot), \xi(\cdot)) $ denote the solution of (\ref{state}) and 
${\mathcal F}^{X, \xi}_t $  denote sigma field generated by  $\{X_s, \xi_s |  s \leq t \}, t \geq 0$.
 This is equivalent to saying that there exists a progressively measurable map 
$\Bar{v} : [0, \ \infty) \times C[ [0, \infty) :\Bar{D}) \times C[ [0, \infty) :\Bar{D}) \to U$ 
such that $v(t) \ = \ \Bar{v}(t, X[0, t], \xi [0, t]), t \geq 0$, where $X[0, t], \xi[0, t] $ denote 
respectively  $\{X_s , 0 \leq s \leq t \}, \{\xi_s , 0 \leq s \leq t \}$. 
Hence by an abuse of notation, we denote the set of feedback controls by all progressively measurable
maps. The following lemma tells that we can restrict ourselves to feedback  controls. Its proof is a straightforward adaptation of Theorem 2.3.4 (a), p.52 of \cite{arapostathis_borkar_ghosh}.
\begin{lemma} Let $(v(\cdot), W(\cdot))$ be an admissible control and $(X(\cdot), \xi(\cdot))$ be a 
solution pair to (\ref{state}) on a filtered probability space $(\Omega, \mathcal {F}, \{\mathcal{F}_t\}, P)$. Then on an augmentation $(\tilde{\Omega}, \tilde{\mathcal{F}}, 
\{\tilde{\mathcal{F}}_t\}, \tilde{P})$ with a $\{ \tilde{\mathcal{F}}_t\}$-Wiener process $\tilde{W}(\cdot)$
and a feedback control $\tilde{v}(\cdot)$ such that $(X(\cdot), \xi(\cdot))$ solves (\ref{state})
for the pair $(\tilde{v}(\cdot), \tilde{W}(\cdot))$ on 
$(\tilde{\Omega}, \tilde{\mathcal{F}}, \{\tilde{\mathcal{F}}_t\}, \tilde{P})$.
\end{lemma}

\subsection{Properties of Controlled RSDEs} We prove some results  about the controlled RSDE
(\ref{state}) which are used in the subsequent sections. To the best of our knowledge these results
are not available the controlled RSDEs we are considering. 

First result is about the equivalence of waek  solution and martingale problem for reflected diffusions. \\
For a feedback  control $v(\cdot)$, we say that the RSDE (\ref{state}) admits a weak solution if 
there exists a filtered probability space $(\Omega, {\mathcal F}, \{\mathcal{F}_t \}, P)$, a
$\{\mathcal {F}_t\}$-Wiener  process $W(\cdot)$ and a pair of $\{\mathcal {F}_t\}$-adapted 
processes $(X(\cdot), \xi(\cdot))$  with a.s. continuous paths such that $X(\cdot)$ is $\overline{D}$-valued,
$\xi(\cdot)$ is non decreasing and satisfy
\begin{eqnarray*}
d X(t) & = & b(X(t), v(t, X[0, t], \xi[0, t]) dt + \sigma(X(t) dW(t) - \gamma(X(t)) d \xi(t)\\
d \xi(t) & = & I_{\{ X(t) \in \partial D \}}d \xi(t), X(0) =x, \xi(0) =0 \ P \ {\rm a.s.} \, .
\end{eqnarray*}
Set
\begin{equation}\label{domain}
{\mathcal H} \ = \ \{ f \in C^2_0 (\overline{D}) | \nabla f \cdot \gamma \geq 0 \ {\rm on} \
\partial D \}
\end{equation}
and 
\begin{equation}\label{generator}
\mathcal{L} f(x, v) \ = \ b(x, v) \cdot \nabla f (x) + \frac{1}{2} {\rm trace}( a(x) \nabla^2 f (x)) ,
f \in \mathcal{D}(\mathcal{L}),
\end{equation}
where the domain $\mathcal{D}(\mathcal{L})$ of the oblique elliptic operator $\mathcal{L}$ 
contains $C^2_{b, \gamma}(\overline{D})$,
the set of all bounded twice continuously differentiable functions satisfying 
$\nabla f \cdot \gamma \geq 0$ on $\partial D$.\\

\noindent {\bf Constrained controlled martingale problem:} A pair of $\{ \mathcal {F}_t\}$-adapted processes $(X(\cdot), \xi(\cdot))$ 
defined on a filtered probability space $(\Omega, {\mathcal F}, \{\mathcal{F}_t \}, P)$
is said solve the constrained controlled martingale problem to  the RSDE (\ref{state}) corresponding
to the admissible control $v(\cdot)$ and initial condition $x \in \overline{D}$ if the following holds.
\begin{itemize}
\item [{(i)}] $X(\cdot)$ is $\overline{D}$-valued and $\xi(\cdot)$ is non decreasing and $X(0) = x, 
\xi(0) =0 $ a.s.
\item [{(ii)}] 
\[
\int^t_0 I_{\{ X(s) \in \partial D \} } d \xi(s) = \xi(t), \ P \ {\rm a.s. \ for\ all} \  t \geq  0,
\]
\item [{(iii)}] For all $ f \in \mathcal {H}$,
\[
M_f (t) \ = \ f(X(t)) - \int^t_0 \mathcal {L} (X(s), v(s)) ds + \int^t_0 \nabla f \cdot \gamma(X(s)) d \xi (s),
\ t \geq 0
\]
is an $\{ \mathcal{F}_t\}$-martingale in $(\Omega, \mathcal{F}, P)$.
\end{itemize}
\begin{theorem} For a feedback  control $v(\cdot)$, the pair of processes $(X(\cdot), \xi(\cdot))$ 
defined on a filtered probability space $(\Omega, {\mathcal F}, \{\mathcal{F}_t \}, P)$ solves the 
constrained controlled martingale problem iff there exists a filtered probability space \\   
$(\tilde{\Omega}, \tilde{{\mathcal F}}, \{ \tilde{\mathcal{F}}_t \}, \tilde{P})$ and a pair of 
processes $(\tilde{X}(\cdot), \tilde{\xi}(\cdot))$ which is a weak solution to (\ref{state}) 
such that $(X(\cdot), \xi(\cdot))$  and $(\tilde{X}(\cdot), \tilde{\xi}(\cdot))$ agree in law.
\end{theorem}
\begin{proof} Suppose  $(X(\cdot), \xi(\cdot))$ solves the constrained controlled martingale problem.
Hence the law  of $X(\cdot)$ solves the corresponding submartingale problem. Now using Theorem 1 of \cite{Kavithaetal2014}, there exists a filtered probability space 
$(\tilde{\Omega}, \tilde{\mathcal {F}}, \{ \tilde{ \mathcal{F}}_t \}, \tilde{P})$ and 
$\{ \tilde{ \mathcal{F}}_t \}$-adapted processes with continuous paths
$(\tilde{X}(\cdot), \tilde{\xi}(\cdot))$ and a Wiener process $\tilde{W}(\cdot)$ such that 
$(\tilde{X}(\cdot), \tilde{\xi}(\cdot))$ is a weak solution to (\ref{state}) and law  of $X(\cdot)$ is same as
law of $\tilde{X}(\cdot)$. Now since (\ref{state}) has a unique weak solution, law of $(X(\cdot), \xi(\cdot))$ 
equals the law  of $(\tilde{X}(\cdot), \tilde{\xi}(\cdot))$. Converse follows from It$\hat{\rm o}$'s 
formula.
\end{proof}
\begin{remark}
Under suitable $C^2$ smoothness assumption on the domain and bounded continuity
assumption on  direction of reflection $\gamma$, the equivalence is shown in \cite{stroock_varadhan}.
The case of domains with piecewise smooth boundaries and with constant direction of
reflections is treated in  \cite{Budhiraja}. 
\end{remark}

For an admissible control $v(\cdot)$, if $(X(\cdot), \xi(\cdot))$ denote a unique weak solution 
pair to the RSDE (\ref{state}) on  $(\Omega, {\mathcal F}, \{\mathcal F_t\}, P)$ and $\tau$  a
$\{\mathcal {F}_t\}$-stopping time, then ${\mathcal F}_{\tau}$ is finitely generated and hence using 
Theorem 1.3.4, p.34 of \cite{stroock_varadhan}, it follows that regular conditional probability 
distribution (rcpd) $P_{\omega}$ of $P$ given ${\mathcal F}_\tau$ exists.  
Now we prove a result analogous to Lemma 2.3.7 of \cite{arapostathis_borkar_ghosh}.
\begin{lemma}Let $(X(\cdot), \xi(\cdot))$  denote a weak solution pair corresponding to an admissible
feedback  control $v(\cdot)$ and defined on  $(\Omega, {\mathcal F}, \{\mathcal F_t\}, P)$ 
and $\tau$ be a finite   $\{\mathcal {F}_t\}$-stopping time. Then the conditional law $\mu_\omega$ 
of the process $X(\tau + \cdot)$ given  ${\mathcal F}_\tau$ is a.s. the law of the process 
$X_{\omega}(\cdot)$, where $X_{\omega}(\cdot)$
is a unique weak solution to the RSDE (\ref{state}) on a probability space 
$(\Omega_{\omega}, {\mathcal F}_{\omega}, \{\mathcal {F}_{\omega, t} \}, P_{\omega})$ for an admissible
control given by $v_{\omega} (t) = v (t + \tau(\omega), X[0, \tau(\omega) + t], \xi[0,\tau(\omega) + t] ),
 t \geq 0$.
\end{lemma}

\begin{proof} For $f \in {\mathcal H}$, since 
\[
M_t \ = \ f(X_t) - f(X_0) - \int^t_0 {\mathcal L} (X_s, v_s) ds + \int^t_0 \nabla f \cdot \gamma (X_s) d \xi_s
, \ t \geq 0,
\]
where ${\mathcal L}$ is given by (\ref{generator})
is an $\{\mathcal {F}_t\}$-martingale on $(\Omega, {\mathcal  F}, P)$, it follows from Theorem 1.2.10, p.28
of \cite{stroock_varadhan} that there exist a $P$-null set $N$ such that for $\omega \notin N$,
$M^{\tau (\omega)}_f (t) = M_t - M_{t \wedge \tau(\omega)} , t \geq 0$ is a Martingale with respect to 
$\{\mathcal {F}_t\}$ on $(\Omega, {\mathcal F}, P_{\omega})$. 
Hence under $P_{\omega}$, 
\[
M^{\tau (\omega)}_f (t) \ = \ f(X_t) - f(X_{\tau(\omega)}) - \int^t_{\tau(\omega)} {\mathcal L} f(X_s, v_s) ds
+ \int^t_{\tau (\omega)} \nabla f \cdot \gamma (X_s) d \xi_s ,  t \geq \tau(\omega) 
\]
is a Martingale under $P_{\omega} , \omega \notin N$. 
i.e.,
\begin{eqnarray*}
M^{\tau (\omega)}_f (t) & = & f(X_t) - f(X_{\tau(\omega)}) - \int^t_0 {\mathcal L} 
f(X(\tau(\omega) + s, v_{s + \tau(\omega)}) ds \\
&& + \int^t_0\nabla f \cdot \gamma (X_{s + \tau (\omega)}) d \xi_{s + \tau(\omega)} ,  t \geq 0 
\end{eqnarray*}
is a Martingale under $P_{\omega} , \omega \notin N$. i.e. $(X_{\omega} (\cdot) , \xi_{\omega}(\cdot)) 
:= (X(\cdot + \omega), \xi(\cdot) + \tau(\omega) - \xi(\tau(\omega))$ solves the constrained controlled 
martingale problem for the admissible control $v_{\omega}$ and initial distribution $X(\tau (\omega))$. This completes the 
proof. 
\end{proof}
Now we give a characterization for recurrence of the RSDE (\ref{state}) corresponding to a stationary
Markov control in the following lemma.
\begin{lemma} Let $v (\cdot) \in \mathcal{S}$ and $X(\cdot)$ be a solution to the RSDE (\ref{state})
corresponding to $v(\cdot)$ and $B$ be a ball  in $D$. Then
$X(\cdot)$ is recurrent iff the PDE
\begin{eqnarray}\label{pdeexterior}
\mathcal{L} \varphi (x, v(x)) & = & 0, \ {\rm in} \ \Bar{B}^c, \nonumber \\ 
\varphi & \equiv & 1 \ {\rm on}\ \partial B, \ \nabla \varphi \cdot \gamma \equiv 0 \ {\rm on} \ \partial D.
\end{eqnarray}
has a unique non negative bounded solution in $W^{2, d +1}_{loc} (\Bar{B}^c) \cap C(B^c)$.
\end{lemma}

\begin{proof} Note that $\varphi \equiv 1$ is always a positive bounded solution of (\ref{pdeexterior}) in 
$W^{2, d +1}_{loc} (\Bar{B}^c) \cap C(B^c)$. Also an application of It$\hat{\rm o}$-Dynkin formula 
and Fatou's lemma implies that any bounded non negative solution  
$\varphi \in W^{2, d +1}_{loc} (\Bar{B}^c) \cap C(B^c)$ satisfies
\[
\varphi (x) \geq P_x (\tau(\Bar{B}^c ) < \infty ) ,  x \in \overline{D}.
\]
Hence the result follows, since non degeneracy of the RSDE implies that $X(\cdot) $ recurrent iff 
it is $B$-recurrent for some ball  $B$ in $D$. 

\end{proof}

\subsection{Notations} In this subsection, we introduce  various frequently used notations in this
paper. 
We denote $\displaystyle{ \sup_{v,  x} |r(x, v)|}$ by $\|r\|_{\infty}.$
For $\varphi \in C_b(\overline{D}),$ the space of all real-valued  bounded continuous functions,
we denote for each $B$, a Borel subset of $\overline{D}$,
\[
\| \varphi \|_{\infty, B} \ = \ \sup_{x \in B} | \varphi (x)|, \ \|\varphi\|_{\infty} = \sup_{x \in \overline{D}}
| \varphi(x)|.
\]

For a Banach space $\mathcal{X}$ with norm $\| \cdot \|_{\mathcal{X}}, \ 1 \leq p <\infty$, define
for $\kappa \geq 0$
\[
L^p( \kappa, T ; \mathcal{X}) \ = \ \{ \varphi : (\kappa  , \ T) \to \mathcal{X} |  \varphi 
\ {\rm is\ Borel\ measurable\ and} \ \int^T_{\kappa} \| \varphi (t) \|^p _{\mathcal{X}} \, dt  < \infty \}
\]
with the norm 
\[
\|\varphi\|_{p; \mathcal{X}} \ = \ \Big[ \int^T_{\kappa} \|\varphi (t) \|^p_{\mathcal{X}} \, dt 
\Big]^{\frac{1}{p}}.
\]

The norm of the Banach space ${L^\infty ((\kappa, 1) \times D)}$ is denoted by 
$\| \cdot\|_{\infty ; (\kappa, 1) \times D}$. 

 $ C^{\infty}_c ( (\kappa, 1) \times D)$  denotes  the space of all functions in 
$C^{\infty} ((\kappa, 1) \times D)$  which are compactly supported. The spaces 
$ C^{\infty}_c ( (\kappa, 1] \times \overline{D}), \  C^{\infty}_c ( [\kappa, 1] \times \overline{D})$
are similarly defined.

For $\kappa < T < \infty$ and an open bounded set  $B$ in $\mathbb{R}^d, \ 
H^{\beta/2, \beta}( [\kappa, T] \times \overline{B}), \kappa \geq 0$,  denotes the set of all
continuous functions $\varphi(t,x)$ in $[\kappa, T] \times \overline{B}$ together with all the derivatives of the from $D^r_tD^s_x\varphi(t,x) $ for $2r+s< \beta$, have a finite norm
$$\|\varphi\|_{H(\beta); [\kappa, T] \times \overline{ B}}
=\| \varphi\|_{\infty ; [\kappa, T] \times \overline{B}}
+H^\beta_{[\kappa, T] \times \overline{B}}(\varphi)
+\sum_{j=1}^{[\beta]}H^j_{[\kappa, T] \times \overline{B}}(\varphi),$$ where
\interdisplaylinepenalty=0
\begin{eqnarray*}
&&H^j_{[\kappa, T] \times \overline{B}}(\varphi)=\sum_{2r+s=j}\| D^r_tD^s_x 
\varphi \|_{\infty ; [\kappa, T] \times \overline{B}} \\
&&H^\beta_{[\kappa, T] \times \overline{B}}(\varphi)=H^\beta_{x, [\kappa, T] \times \overline{B}}(\varphi)
+H^{\beta/2}_{t, [\kappa, T] \times \overline{B}}(\varphi) \\
&&H^\beta_{x,(\kappa, T) \times B}(\varphi)=\sum_{2r+s=[\beta]}
H^{(\beta-[\beta])}_{x, [\kappa, T] \times \overline{B}} (D^r_tD^s_x \varphi ) \\
&&H^{\beta/2}_{t, [\kappa, T] \times \overline{B}}(\varphi)=\sum_{0<\beta-2r-s<2}
H^{(\frac{\beta-2r-s}{2})}_{t, [\kappa, T] \times \overline{B}} (D^r_tD^s_x \varphi ) \\
&&H^{(\alpha)}_{x, [\kappa, T] \times \overline{B}} (\varphi)=\displaystyle{\sup_{(t,x),(t,\bar{x})\in 
[\kappa, T] \times \overline{B}}
\dfrac{|\varphi(t,x)-\varphi(t,\bar{x})|}{\|x-\bar{x}\|^\alpha}}, \; 0<\alpha <1, \\
&&H^{(\alpha)}_{t, [\kappa, T] \times \overline{B}} (\varphi)=\displaystyle{\sup_{(t,x),(\bar{t},x)\in 
[\kappa, T] \times \overline{B}}\dfrac{|\varphi(t,x)-\varphi(\bar{t},x)|}{|t-\bar{t}|^\alpha}}, \; 0<\alpha <1.
\end{eqnarray*}
We denote 
\[
C^{\frac{\beta}{2}, \beta}( [\kappa, T] \times \overline{B}) \ 
= \ \{ \varphi \in C( [\kappa, T] \times \overline{B})
| \varphi \in C^{\beta/2, \beta}( [\kappa, T] \times K), \ {\rm for\ some\ compact\ subset\ of} \ \overline{B}
\} .
\]

The space $\mathcal{W}^{1,2,p}( (\kappa, T) \times \overline{D})), \kappa \geq 0$, denotes the set of all
$\varphi \in L^p ( \kappa, T ; W^{2, p}(\overline{D}))$ such that 
$\frac{\partial \varphi}{\partial t} \in L^p ( (\kappa, T; L^p(\overline{D}))$ with the norm
given by
\[
\|\varphi\|^p_{1,2,p; W^{2, p}(\overline{D})} \ = \ 
\| \varphi \|^p_{p; W^{2,p}(\overline{D})} + \| \frac{\partial \varphi}{\partial t} \|^p_{p; L^p(\overline{D})}
, \ 1 \leq p < \infty.
\]

Also the local Sobolev spaces $\mathcal{W}^{1,2, p}_{\rm loc}((\kappa , T) \times \overline{D})$ 
are defined by
\begin{eqnarray*}
&&\mathcal{W}^{1,2, p}_{{\rm loc}}(\kappa, T) \times \overline{D})\\& = & 
\Big \{ \varphi : (\kappa, T) \times \overline{D} \to
\mathbb{R}| \, \varphi \ {\rm is\ measurable\ and} \ \varphi \in W^{1,2,p}( (\kappa, T) \times K),
 \\&& {\rm for\ some}\ K\ {\rm is \ a\ compact\ subset\ of }\ \overline{D} \Big \}.
\end{eqnarray*}

For any  domain $B$ in $\overline{D}$, define $$W^{1,2, p} ( (\kappa, T) \times B) \ = \ 
\Big\{ \varphi : (\kappa , T ) \times B \to \mathbb{R} \Big| \|\varphi\|_{1,2,p; (\kappa, T) \times B} < \infty
\Big \}, $$ where the norm $\| \cdot \|_{1,2,p; (\kappa, T) \times B}$ is defined as 
\interdisplaylinepenalty=0
\begin{eqnarray*}
\|\varphi\|^p_{1,2,p; (\kappa, T) \times B } & = &  \int^T_{\kappa} \int_{B} | \varphi(t, x)|^p  dx dt  + 
\int^T_{\kappa} \int_{B} \Big | \frac{\partial \varphi(t,x)}{\partial t}\Big |^p  dx dt \\
& +& \sum_i \int^T_{\kappa} \int_{B} \Big| \frac{\partial \varphi(t,x)}{\partial x_i} \Big|^p  dx  dt+ 
\sum_{ij} \int^T_{\kappa} \int_{B} \Big| \frac{\partial^2 \varphi(t,x)}{\partial x_i x_j} \Big|^p| dx dt.
\end{eqnarray*}

\section{ Analysis of the Discounted  Cost criterion }
In this section, we  study the  discounted  risk-sensitive control problem
with the state dynamics  (\ref{state}) and cost criterion 
\[
J_{\alpha}^v (\theta,x)= \frac{1}{\theta} \, \ln \,
 E_x^v\left[ e^{\theta\int_{0}^{\infty} e^{-\alpha t} \; r(X_t,v_t) dt}\right].
\]
The $\alpha$-discounted  risk-sensitive control problem is to minimize (\ref{discountedcost}) over all admissible 
controls. We define the so-called `value function' for the cost 
(\ref{discountedcost}) as
\begin{equation}\label{discounted_risk_value}
 \phi_{\alpha}(\theta,x)= \inf_{v\in\mathcal A} J_{\alpha}^v (\theta,x).
\end{equation}
Set
\begin{equation}\label{discounted_cost}
\bar J_{\alpha}^v (\theta,x)= E_x^v\left[ e^{\theta\int_{0}^{\infty} e^{-\alpha t} \; r(X_t,v_t) dt}\right].  
\end{equation}
Since logarithm is an increasing function for fixed $\theta>0$, a minimizer of $\bar J_{\alpha}^v (\theta,x)$ if it exists will be a
minimizer of $J_{\alpha}^v (\theta,x))$. Corresponding to the cost (\ref{discounted_cost}), the value function is defined as 
\begin{equation}\label{discounted_value}
 u_{\alpha}(\theta,x)= \inf_{v\in \mathcal A} \bar J_{\alpha}^v (\theta,x).
\end{equation}
Note that 
\begin{equation}\label{relation}
 \phi_{\alpha}(\theta,x)= \frac{1}{\theta} \ln u_{\alpha}(\theta,x) .
\end{equation}
Since we are dealing with exponential cost we need {\it multiplicative} version of DPP in place of  
additive DPP, see [\cite{borkar}, pp. 53-59].
We mimic the arguments as in \cite{menaldi_robin} to prove DPP for the value function $u_{\alpha}(\theta,x)$.
 
\begin{theorem}[DPP]
Let $\tau$ be any bounded stopping time with respect to the natural filtration of process $X(\cdot)$, i.e., $\{ \mathcal F^X_t\}$. 
Then 
\begin{equation}\label{DPP}
 u_{\alpha}(\theta,x) = \inf_{v(\cdot)} E^v_x\left[ e^{\theta\int_{0}^{\tau} e^{-\alpha t} \; r(X_t,v_t) dt}
u_{\alpha}\left(\theta e^{-\alpha \tau}, X(\tau)\right) \right]. 
\end{equation}
where infimum is taken over all feedback controls.
\end{theorem}
\begin{proof} Note that, given two feedback controls $v_1(t)$ and $v_2(t)$,  $t\geq 0$ and $\tau$ as
above, $v(\cdot)$ defined as
\begin{equation}\label{control_dpp}
 v(t) = v_1(t) I_{\{ t < \tau \} } + v_2(t-\tau) I_{\{ t \geq \tau \} }, \ t\geq 0,
\end{equation}
is also a feedback control. Indeed, we are given pairs of processes 
$(X_1(\cdot), \xi_1(\cdot), v_1(\cdot))$ and
$(X_2(\cdot), \xi_1(\cdot), v_2(\cdot))$ satisfying (\ref{state}) on some, possibly distinct, probability spaces 
$(\Omega_1,\mathcal F_1,P_1)$, $(\Omega_2,\mathcal F_2,P_2)$ respectively, with $v_1(\cdot), v_2(\cdot)$ in
feedback from. Also, $X_1(0)=x$ and the law of 
 $X_2(0)=$ the law of $X_1(\tau)$, where $\tau$ is a prescribed stopping time with respect to the natural filtration of process $X_1(\cdot)$. By augmenting $(\Omega_1,\mathcal F_1,P_1)$ suitably, one can construct a processes $(X(\cdot), \xi(\cdot))$  and $v(\cdot)$
 satisfying (\ref{state}) such that they coincide with $(X_1(\cdot), \xi_1(\cdot))$ and $v_1(\cdot)$ on
 $[0,\tau]$, and $(X(\tau+\cdot), \xi(\tau+\cdot))$ and $v(\tau+\cdot)$ agree in law with $(X_2(\cdot),\xi(\cdot))$ and $v_2(\cdot)$. Also the conditional law of $X(\tau+\cdot)$ of given 
 $\mathcal F_{\tau}$ is the same as its conditional law given $X(\tau)$ and agrees with the conditional
 law of $X(\tau+\cdot)$ given $X_2(0)$ a.s. with respect to the common law of $X_2(0),X(\tau)$.
The above construction uses Lemma 1.2. 

Let $\epsilon >0$. Let $X(\cdot)$ be a process (\ref{state}) controlled by $v(\cdot)$ as above 
with $v_1(\cdot)$ an arbitrary feedback control and $v_2(\cdot)$ an $\epsilon$-optimal 
 feedback control for initial data $X(\tau)$.
By (\ref{discounted_value}) we have
\begin{eqnarray*}
 u_{\alpha}(\theta,x) &\leq& E^v_x \left[ e^{\theta\int_{0}^{\tau} e^{-\alpha t} \; r(X_t,v_t) dt 
+ \theta \int_{\tau}^{\infty} e^{-\alpha t} \; r(X_t,v_t) dt } \right] \\
&=& E^v_x \left[ e^{\theta\int_{0}^{\tau} e^{-\alpha t} \; r(X_t,v_t) dt }
 \times  e^{ \theta e^{-\alpha \tau} \int_{0}^{\infty} e^{-\alpha t} \; r(X_{t+\tau},v_{t+\tau}) dt } \right] \\
&=& E^v_x \left[ e^{\theta\int_{0}^{\tau} e^{-\alpha t} \; r(X_t,v_t) dt }
  E\left[ e^{ \theta e^{-\alpha \tau} \int_{0}^{\infty} e^{-\alpha t} \; r(X_t,v_t) dt }\Big | X(\tau) \right] \right]\\
&\leq& E^v_x \left[ e^{\theta\int_{0}^{\tau} e^{-\alpha t} \; r(X_t,v_t) dt } 
\left( u_{\alpha}\left(\theta e^{-\alpha \tau}, X(\tau)\right) + \epsilon\right) \right] \\
&=& E^v_x \left[ e^{\theta\int_{0}^{\tau} e^{-\alpha t} \; r(X_t,v_t) dt } 
 u_{\alpha}\left(\theta e^{-\alpha \tau}, X(\tau)\right)  \right] + \epsilon 
E^v_x \left[ e^{\theta\int_{0}^{\tau} e^{-\alpha t} \; r(X_t,v_t) dt }\right] .
\end{eqnarray*}
Since $\tau, r $ are bounded and $\epsilon$ is arbitrary we get
$$
u_{\alpha}(\theta,x) \leq \inf_{v(\cdot)} E^v_x \left[ e^{\theta\int_{0}^{\tau} e^{-\alpha t} \; r(X_t,v_t) dt } 
 u_{\alpha}\left(\theta e^{-\alpha \tau}, X(\tau)\right)  \right].
$$
Conversely, Let $\epsilon>0$ and $v(\cdot)$ is an $\epsilon$-optimal feedback control for initial data $X(0)=x$. Then 
\begin{eqnarray*}
 u_{\alpha}(\theta,x)+\epsilon &\geq& E^v_x\left[ e^{\theta\int_{0}^{\tau} e^{-\alpha t} \; r(X_t,v_t) dt 
+ \theta \int_{\tau}^{\infty} e^{-\alpha t} \; r(X_t,v_t) dt } \right] \\
&=& E^v_x\left[ e^{\theta\int_{0}^{\tau} e^{-\alpha t} \; r(X_t,v_t) dt }
  E\left[ e^{ \theta e^{-\alpha \tau} \int_{0}^{\infty} e^{-\alpha t} \; r(X_t,v_t) dt }\Big | X(\tau) \right] \right]\\
&\geq& E^v_x\left[ e^{\theta\int_{0}^{\tau} e^{-\alpha t} \; r(X_t,v_t) dt }
  \inf_{v(\cdot)} E\left[ e^{ \theta e^{-\alpha \tau} \int_{0}^{\infty} e^{-\alpha t} \; r(X_t,v_t) dt }\Big | X(\tau) \right] \right]\\
&=& E^v_x \left[ e^{\theta\int_{0}^{\tau} e^{-\alpha t} \; r(X_t,v_t) dt } 
 u_{\alpha}\left(\theta e^{-\alpha \tau}, X(\tau)\right) \right] .
\end{eqnarray*}
Thus 
$$
u_{\alpha}(\theta,x)+\epsilon \geq  \inf_{v(\cdot)} E^v_x \left[ e^{\theta\int_{0}^{\tau} e^{-\alpha t} \; r(X_t,v_t) dt } 
 u_{\alpha}\left(\theta e^{-\alpha \tau}, X(\tau)\right)  \right].
$$
Letting $\epsilon\longrightarrow 0$ completes the proof.  
\end{proof}

Using dynamic programming heuristics, the HJB equations for discounted cost criterion is given by

\begin{equation}\label{discount_hjb}
\begin{array}{rcl}
\displaystyle{
 \alpha\theta\frac{\partial u_{\alpha}}{\partial \theta} } &=& \displaystyle{ \inf_{v \in U} 
\left[ b(x,v)\cdot 
\nabla u_{\alpha} + \theta r(x,v)u_{\alpha}\right] 
+ \frac{1}{2}trace (a(x)\nabla^2 u_{\alpha})  } \\
\displaystyle{ u_{\alpha}(0,x) } &=& \displaystyle{ 1 \mbox{ on } \overline {D}, \ \ \nabla u_{\alpha}(\theta,x) .\ \gamma(x)=0 \mbox{ on } (0,1) \times \partial D.}
\end{array}
\end{equation}
First we show that  (\ref{discount_hjb}) has unique a solution. 
There are two technical difficulties in solving  the p.d.e. (\ref{discount_hjb}). First is the
singularity in $\theta$ at $0$ and the second is the unbounded non smooth nature of the orthrant.
We circumvent these difficulties by suitable approximation arguments as follows.
For each $m,l \geq 1$ and $0 < \kappa < 1$, consider the p.d.e. 
\begin{equation}
\begin{array}{rcl}\label{D_m}
\displaystyle{
 \alpha\theta\frac{\partial u^{\kappa}_{\alpha,lm}}{\partial \theta} } &=& 
\displaystyle{ \inf_{v \in U} \left[ b(x,v)\cdot 
\nabla u^{\kappa}_{\alpha,lm} + \theta r(x,v)u^{\kappa}_{\alpha,lm}\right]+ \frac{1}{2} \, 
trace (a(x)\nabla^2 u^{\kappa}_{\alpha,lm}) } \\
\displaystyle{ u^{\kappa}_{\alpha,lm}(\kappa,x) } &=& \displaystyle{ e^{\frac{\kappa \|r\|_{\infty}}{\alpha}} \mbox{ on }  \overline D_{lm} , \ \ 
\nabla u^{\kappa}_{\alpha,lm}. \, \gamma=0 \mbox{ on } (\kappa,1)\times \partial D_{lm}.}
\end{array}
\end{equation}

\begin{lemma} \label{lemma_estimate}
Assume (A1) and (A2).  Then the  p.d.e. (\ref{D_m}) has a unique solution $ u^{\kappa}_{\alpha,lm}\in  H^{\frac{3}{2},3}([\kappa, 1] \times \overline{D}_{lm}) $, and 
 \begin{eqnarray}
\label{bound1} & \|u^{\kappa}_{\alpha,lm}\|_{\infty; [\kappa, 1] \times \overline{D}_{lm}} \leq e^{\frac{\theta \|r\|_{\infty}}{\alpha}},  \  \mbox{for all}  \; \kappa >0 , \ m,l \geq 1,\ 
   \\
\label{bound2} & \Big\| \frac{\partial u^{\kappa}_{\alpha,lm}}{\partial \theta}  
\Big\|_{\infty ; [\kappa, 1] \times \overline{D}_{lm}} \leq
3 e^{\frac{(\theta+3)\|r\|_{\infty}}{\alpha}} \frac{\|r\|_{\infty}}{\alpha}, \, \ \mbox{for all} \ \kappa >0 , 
\ m,l \geq 1 .
\end{eqnarray} 
\end{lemma}
\begin{proof} For the existence and uniqueness result we  use  Theorem 7.4 from [\cite{lady},  p. 491]. 
Set
  $$
  \theta=1-t, \ u^{\kappa}_{\alpha,lm}(\theta,x)= u(t,x).
  $$
  Then equation (\ref{D_m}) reduce to
  \begin{eqnarray*}
        \displaystyle{
     -\alpha(1-t)\frac{\partial u}{\partial t} } &=& 
    \displaystyle{ \inf_{v \in U} \left[ b(x,v)\cdot 
    \nabla u + (1-t) r(x,v)u\right]  +\frac{1}{2} \,  \mbox{trace} (a(x)\nabla^2 u) } \\
    \displaystyle{ u(1-\kappa,x) } &=& \displaystyle{ e^{\frac{\kappa \|r\|_{\infty}}{\alpha}}, 
\mbox{ for } x \in \overline{D}_{m} \ \ , \ \ 
    \nabla u(t,x).\gamma(x)=0 \mbox{ on } (0,1-\kappa)\times \partial D_{m} }  \\ 
       \end{eqnarray*}
Rewrite the above equation as
\begin{eqnarray}\label{conormal4Dml}
        \displaystyle{
     0 } &=& 
    \displaystyle{ \frac{\partial u}{\partial t} + \inf_{v \in U} \left[ \frac{ b(x,v)\cdot 
    \nabla u}{\alpha(1-t)} + \frac{1}{\alpha} r(x,v)u\right]  +\frac{1}{2} \, 
    \mbox{trace} \left(\frac{a(x) }{\alpha(1-t)} \nabla^2 u \right) }  \nonumber \\
   \displaystyle{ u(1-\kappa,x) } &=& \displaystyle{ e^{\frac{\kappa \|r\|_{\infty}}{\alpha}}, \mbox{ for } x \in \overline{D}_{lm} \ \ , \ \ 
    \nabla u(t,x).\gamma(x)=0 \mbox{ on } (0,1-\kappa)\times \partial D_{lm} }. \nonumber \\ 
        \end{eqnarray}
 Set
 \begin{eqnarray}
  b(t,x,u,p) &=& \inf_{v\in U} \left[ \frac{b(x,v)\cdot p}{\alpha(1-t)} +\frac{1}{\alpha} r(x,v) u\right] \label{hamiltonian}\\
  a_{ij}(x,t)&=&\frac{a_{ij}(x)}{2\alpha(1-t)} \nonumber \\
 T &=& 1-\kappa \nonumber \\
 Q_T&=& D_{lm} \times [0,T]\nonumber \\
  \psi_0(x)&=& e^{\frac{\kappa \|r\|_{\infty} }{\alpha}} \nonumber.
 \end{eqnarray}
Note that $ b(t,x,u,p)$ and $a_{ij}(x,t)$ are Lipschitz continuous in $x$, since $ b(x,v),$ $ r(x,v), a_{ij}(x)$ are Lipschitz continuous in the first argument uniformly with respect to the second.

 Therefore from [\cite{lady}, Theorem 7.4, p. 491] it follows that (\ref{conormal4Dml}) has a unique solution in $  H^{\frac{3}{2},3}([\kappa,1]\times \overline{D}_{lm})$.

Let  $v(\cdot)$ be an admissible control and  $X(\cdot)$ be the process given by
 \[
 \begin{array}{rcl}
 \displaystyle{
  dX_t} &=& \displaystyle{ b(X_t,v_t)dt+\sigma(X_t) dW_t - \gamma(X_t)d\xi_t } \\
 \displaystyle{ d\xi_t }&=&  \displaystyle{  I_{\{X_t\in \partial D_{lm}\}} d\xi_t  } \\
 \displaystyle{ \xi_0 } &=& \displaystyle{  0,\ \ \ \ X_0 = x\in \overline{D}_{lm}\, . } 
 \end{array}
 \]
  Applying It$\hat{{\rm o}}$'s formula to $\displaystyle {
 e^{\int_{0}^{t} \theta_s r(X_s,v_s) ds} \, u^{\kappa}_{\alpha,lm}(\theta_t , X_t) }, \theta_t \ = \ 
 \theta e^{-\alpha t},$ we get
 \begin{eqnarray*}
  d\left( e^{\int_{0}^{t} \theta_s r(X_s,v_s) ds} u^{\kappa}_{\alpha,lm}(\theta_t , X_t)\right) &=& 
  e^{\int_{0}^{t} \theta_s r(X_s,v_s) ds} du^{\kappa}_{\alpha,lm}(\theta_t , X_t) \\
 &&+  \theta_t u^{\kappa}_{\alpha,lm}(\theta_t , X_t) e^{\int_{0}^{t} \theta_s r(X_s,v_s) ds} r(X_t,v_t) dt,
 \end{eqnarray*}
 where
 \begin{eqnarray*}
 du^{\kappa}_{\alpha,lm}(\theta_t , X_t) &=&  (\nabla u^{\kappa}_{\alpha,lm}(\theta_t, X_t))^\bot 
 \sigma(X_t) dW_t - \left[\gamma(X_t)\cdot \nabla u^{\kappa}_{\alpha,lm}
 (\theta_t, X_t)\right] I_{\{X_t\in\partial D_{lm}\}} d\xi_t \\ 
 && + \,  \Big[ {\mathcal L} \, u^{\kappa}_{\alpha, lm} (\theta_t, X_t, v_t) \, 
 - \alpha \theta_t \frac{\partial}{\partial \theta} u^{\kappa}_{\alpha,lm}(\theta_t , X_t) \Big] dt .
 \end{eqnarray*}
 Using the fact that $u^{\kappa}_{\alpha,lm}$ satisfy the equation (\ref{D_m}), we have 
 \[
  u^{\kappa}_{\alpha,lm}(\theta, x) \ \leq \ 
 E_x^v\Big[ e^{\frac{\kappa \|r\|_{\infty}}{\alpha}}
 e^{\int_{0}^{T_{\kappa}} \theta e^{-\alpha s} r(X_s,v_s) ds}\Big] \, , 
 \]
 where $\displaystyle{ T_{\kappa}=\frac{ln(\frac{\theta}{\kappa})}{\alpha} }$. 
 Repeating the above argument with a minimizing selector in (\ref{D_m}), we get 
 \begin{equation}\label{representation1}
  u^{\kappa}_{\alpha,lm}(\theta, x) \ = \ \inf_{v(\cdot)} 
 E_x^v\Big[ e^{\frac{\kappa \|r\|_{\infty}}{\alpha}}
 e^{\int_{0}^{T_{\kappa}} \theta e^{-\alpha s} r(X_s,v_s) ds}\Big] \,.
 \end{equation}
 From (\ref{representation1}), we have
 \begin{eqnarray*}
  |u^{\kappa}_{\alpha,lm}| \leq E_x^v\Big[ e^{\frac{\kappa \|r\|_{\infty}}{\alpha}}
 e^{\int_{0}^{T_{\kappa}} \theta e^{-\alpha s} r(X_s,v_s) ds} \,  \Big] 
 \leq  e^{\frac{\kappa \|r\|_{\infty}}{\alpha}} e^{\|r\|_{\infty} \frac{(\theta -\kappa)}{\alpha} },
 \end{eqnarray*}
 which proves the estimate (\ref{bound1}). 

We mimic the arguments of [\cite{anup_borkar_suresh}, Theorem 3.1], to prove the estimate (\ref{bound2}).
For $\epsilon$ with 
$|\epsilon|$ sufficiently small, set
$$
T^{\epsilon}_{\kappa} = \frac{1}{\alpha} \log\left(\frac{\theta + \epsilon}{\kappa}\right).
$$ 
Now consider for each $v(\cdot)$ admissible
\begin{equation}\label{inter}
 \left.
\begin{array}{lll}
 \displaystyle{ \left| E^v_x\left[ e^{ (\theta+\epsilon) \int_{0}^{T^{\epsilon}_{\kappa}}  e^{-\alpha t} r(X_t, v_t) dt }
\right] - E^v_x\left[ e^{ \theta \int_{0}^{T_{\kappa}}  e^{-\alpha t} r(X_t, v_t) dt }\right] \right| } \\
\leq \displaystyle{ \left| E^v_x\left[ e^{ (\theta+\epsilon) \int_{0}^{T^{\epsilon}_{\kappa} } e^{-\alpha t} r(X_t, v_t) dt }\right]
 - E^v_x\left[ e^{ \theta \int_{0}^{T^{\epsilon}_{\kappa}}  e^{-\alpha t} r(X_t, v_t) dt }\right] \right| } \\
\ \ + \displaystyle{ \left| E^v_x\left[ e^{ \theta \int_{0}^{T^{\epsilon}_{\kappa} } e^{-\alpha t} r(X_t, v_t) dt }\right] -
E^v_x\left[ e^{ \theta \int_{0}^{T_{\kappa}}  e^{-\alpha t} r(X_t, v_t) dt } \right]\right| }
\end{array}
\right\}.
\end{equation}
Now 
\begin{equation}\label{inter1}
 \left.
\begin{array}{lll}
 \displaystyle{ \left| E^v_x\left[ e^{ (\theta+\epsilon) \int_{0}^{T^{\epsilon}_{\kappa} } e^{-\alpha t} r(X_t, v_t) dt }\right]
 - E^v_x\left[ e^{ \theta \int_{0}^{T^{\epsilon}_{\kappa}}  e^{-\alpha t} r(X_t, v_t) dt }\right] \right| } \\
\leq \displaystyle{  E^v_x\left[ e^{ \theta \int_{0}^{T^{\epsilon}_{\kappa} } e^{-\alpha t} r(X_t, v_t) dt } 
\times \left| e^{ \epsilon \int_{0}^{T^{\epsilon}_{\kappa}}  e^{-\alpha t} r(X_t, v_t) dt } -1 \right| \right] } \\
\leq \displaystyle{   e^{ \frac{\theta \|r\|_{\infty} } {\alpha} \left(1-\frac{\kappa} {\epsilon +\theta}\right)  } 
\times E^v_x \left| e^{ \epsilon \int_{0}^{T^{\epsilon}_{\kappa}}  e^{-\alpha t} r(X_t, v_t) dt } -1 \right|  } \\
\leq e^{\frac{(\theta+\epsilon)\|r\|_{\infty}}{\alpha}} \frac{\|r\|_{\infty}}{\alpha} |\epsilon|
\end{array}
\right\},
\end{equation}
and 
\begin{eqnarray*}
 \left.
\begin{array}{lll}
 \displaystyle{ \left| E^v_x\left[ e^{ \theta \int_{0}^{T^{\epsilon}_{\kappa} } e^{-\alpha t} r(X_t, v_t) dt }\right] -
E^v_x\left[ e^{ \theta \int_{0}^{T_{\kappa}}  e^{-\alpha t} r(X_t, v_t) dt } \right] \right| } \\
\leq \displaystyle{  E^v_x\left[ e^{ \theta \int_{0}^{T_{\kappa} } e^{-\alpha t} r(X_t, v_t) dt } 
\times \left| e^{ \theta \left| \int_{ T_{\kappa} }^{T^{\epsilon}_{\kappa}}  e^{-\alpha t} r(X_t, v_t) dt \right|}  -1 \right| 
 \right] } \\
\leq \displaystyle{ e^{\frac{\theta\|r\|_{\infty}}{\alpha}}\left[ e^{\frac{\theta\|r\|_{\infty}}{\alpha} \left|
 e^{-\alpha T_{\kappa} } - e^{-\alpha T^{\epsilon}_{\kappa}} \right|}  - 1 \right] }\\
= \displaystyle{ e^{\frac{\theta\|r\|_{\infty}}{\alpha}} \left[ e^{ \frac{\|r\|_{\infty}}{\alpha} \left|
\frac{\kappa \epsilon} {\theta +\epsilon} \right| } - 1\right].  }
\end{array}
\right.
\end{eqnarray*}
Note that for each $\theta >0$, when $\epsilon$ is positive, then $ \displaystyle{ \frac{\kappa\epsilon}{\theta+\epsilon} 
\leq 1 }$ and for $\epsilon <0$ we can choose a $0<\epsilon_{\theta} <1$ such that $ \displaystyle{ \frac{\kappa\epsilon}
{\theta+\epsilon} \leq 2 }$ whenever $|\epsilon| \leq \epsilon_{\theta}$. Hence we have
\begin{equation}\label{inter2}
 \left.
\begin{array}{lll}
 \displaystyle{ \left| E^v_x\left[ e^{ \theta \int_{0}^{T^{\epsilon}_{\kappa} } e^{-\alpha t} r(X_t, v_t) dt }\right] -
E^v_x\left[ e^{ \theta \int_{0}^{T_{\kappa}}  e^{-\alpha t} r(X_t, v_t) dt } \right] \right| } \\
\leq \displaystyle{ e^{\frac{\theta\|r\|_{\infty}}{\alpha}} \left[ e^{ \frac{2 \|r\|_{\infty}}{\alpha} |\epsilon|} - 1\right] 
 \ \mbox{ whenever } |\epsilon|\leq \epsilon_{\theta}  } \\
\leq \displaystyle{ e^{\frac{\theta\|r\|_{\infty}}{\alpha}} \frac{2\|r\|_{\infty}}{\alpha} |\epsilon|
 e^{ \frac{2 \|r\|_{\infty}}{\alpha} |\epsilon|}  \ \mbox{ whenever } |\epsilon|\leq \epsilon_{\theta} } \\
= \displaystyle{ 2 e^{\frac{(\theta+2)\|r\|_{\infty}}{\alpha}} \frac{\|r\|_{\infty}}{\alpha} |\epsilon|
   \ \mbox{ whenever } |\epsilon|\leq \epsilon_{\theta} }
\end{array}
\right\}.
\end{equation}
From (\ref{representation1}), (\ref{inter}), (\ref{inter1}) and (\ref{inter2}) we have
\begin{eqnarray*}
 | u^{\kappa}_{\alpha,lm}(\theta+\epsilon,x)- u^{\kappa}_{\alpha,lm}(\theta,x)| &\leq& 
e^{\frac{\kappa\|r\|_{\infty}}{\alpha}} \sup_{v(\cdot)}\left| E^v_x\left[ e^{ (\theta+\epsilon) \int_{0}^{T^{\epsilon}_{\kappa}}  
e^{-\alpha t} r(X_t, v_t) dt }\right] \right. \\
&& \ \ \ \ \left. - E^v_x\left[ e^{ \theta \int_{0}^{T_{\kappa}}  e^{-\alpha t} r(X_t, v_t) dt }\right] \right| \\
&\leq&  3 e^{\frac{(\theta+3)\|r\|_{\infty}} {\alpha}} \frac{\|r\|_{\infty}}{\alpha} |\epsilon|
   \ \mbox{ whenever } |\epsilon|\leq \epsilon_{\theta}.
\end{eqnarray*}
This completes the proof of the lemma.
\end{proof}

\begin{theorem}\label{Thm_discount_hjb}
 Assume (A1) and (A2). Then equation (\ref{discount_hjb}) has a solution 
 $u_{\alpha} \in  W^{1,2,p}_{loc}((0, \, 1) \times \overline{D} ), \, p \geq 2$.
\end{theorem}

\begin{proof} Let $C  $ be an  open bounded  set with $C^\infty$ boundary such that  
$\overline{C} \subseteq \overline{D}.$ 
Let $N$ be a positive integer such that $$C \subseteq \overline{D}_{lm}, \ {\rm for \ all \ } m\geq 1, \
l \geq N . $$ 
From Lemma \ref{lemma_estimate}, p.d.e. (\ref{D_m}) has a unique solution 
 $u^{\kappa}_{\alpha,lm} \in     H^{\frac{3}{2},3}([\kappa,1]\times \overline{D}_{lm}) $ and   
\begin{equation}
 \|u^{\kappa}_{\alpha,lm}\|_{\infty; (\kappa, \ 1) \times D_{lm}}\leq e^{\frac{\theta \|r\|_{\infty}}{\alpha}}, \ 
\forall\ \kappa >0 \ \& \ m,l\geq 1. \nonumber
\end{equation}
Thus from Theorem 9.11, p.235 of \cite{gilbarg_trudinger}, we get
  \begin{equation}\label{sobolevestimate}
  \|u^{\kappa}_{\alpha,lm}\|_{1,2,p ; ((\kappa,1) \times \overline{C}) } < K, \ {\rm for \ all \ }  m \geq 1,
 l\geq N,  p \geq 2,
 \end{equation}
 where $K$ does not depend on $l$ and $m$. 
Now choose a sequence of open domains $\{ C_n\}$  from $\overline{D}$  such that 
$\cup_n \overline{C}_n = \overline{D}$. Now by a standard diagonalization procedure there exists 
$u^{\kappa}_{\alpha} \in W^{1,2,p}_{loc} ( (\kappa , 1) \times \overline{D})$ such that 
 along a  subsequence in $l \to \infty$,
\begin{equation}\label{weak}
u^{\kappa}_{\alpha,lm} \longrightarrow u^{\kappa}_{\alpha} \ \ \ \mbox{weakly in }  
\ W^{1,2,p}((\kappa,1) \times C) .
\end{equation}
Now from (\ref{sobolevestimate}), we have 
 \begin{equation}\label{sobolevestimate1}
  \|u^{\kappa}_{\alpha,m}\|_{1,2,p ; ((\kappa,1) \times \overline{C}) } < K, \ {\rm for \ all \ }  m \geq 1.
 \end{equation}
Now by repeating  the diagonalization argument there exists $u^\kappa_\alpha \in 
W^{1,2,p}_{loc} ( (\kappa , 1) \times \overline{D})$ such that along a subsequence in $m \to \infty$
\begin{equation}\label{weak1}
u^{\kappa}_{\alpha,m} \longrightarrow u^{\kappa}_{\alpha} \ \ \ \mbox{weakly in }  
\ W^{1,2,p}((\kappa,1) \times \overline{C}) .
\end{equation}
Using  parabolic version of the Morrey's lemma, see [\cite{WuYin}, pp.26-27],  $W^{1,2,p}((\kappa,1)\times C)$ is compactly embedded in  $H^{\frac{\hat{\alpha}}{2}, \hat{\alpha} } ([ \kappa, 1 ] \times 
\overline{C} ), 0 < \hat{\alpha} < 2 - \frac{d+2}{p}  $. 
Hence along a subsequence of $l \to \infty, m \to \infty$, we get 
\begin{equation}\label{C0}
\lim_{m \to \infty} \lim_{l \to \infty}u^{\kappa}_{\alpha,lm} \ = \  u^{\kappa}_{\alpha} \ \
 \mbox{where\ the\  convergence\ is\ in } 
H^{\frac{\hat{\alpha}}{2}, \hat{\alpha} } ([ \kappa, 1 ] \times \overline{C} ) .
\end{equation}

Now (\ref{C0}) implies (along a subsequence in $l, m \to \infty$)
\begin{equation}\label{nonlinear}
 \lim_{m \to \infty} \lim_{l \to \infty} \inf_{v} \left[ b(x,v)\cdot 
\nabla u^{\kappa}_{\alpha,m} + \theta r(x,v)u^{\kappa}_{\alpha,m}\right]
\ = \  \inf_{v} \left[ b(x,v)\cdot 
\nabla u^{\kappa}_{\alpha} + \theta r(x,v)u^{\kappa}_{\alpha}\right] \mbox{ in } [\kappa,1] \times 
\overline{C}.
\end{equation}

By letting (along a subsequence) 
$l \to \infty$ and then  $m\longrightarrow \infty$  in (\ref{D_m}), with the help of (\ref{weak}) and (\ref{nonlinear}), we get
\begin{equation}
 \alpha\theta\frac{\partial u^{\kappa}_{\alpha}}{\partial \theta}=  \inf_{v} \left[ b(x,v)\cdot 
\nabla u^{\kappa}_{\alpha} + \theta r(x,v)u^{\kappa}_{\alpha}\right] +\frac{1}{2}
\mbox{trace}(a(x)\nabla^2 u^{\kappa}_{\alpha}) \mbox{ in } (\kappa,1)\times D 
\end{equation}
in the sense of distribution and $u^{\kappa}_{\alpha} \in W^{1,2, p}([\kappa,1]\times C) $ for any
compact subset $\overline{C}$ of $\overline{D}$ with $C^\infty$ smooth boundary.  
Also from (\ref{C0}) it follows that $u^{\kappa}_{\alpha}(\kappa,x) = e^{\frac{\kappa \|r\|_{\infty}}{\alpha}}$.
\\
 \paragraph{}  Since
\[
\nabla u^\kappa_{\alpha, m} \cdot \gamma \equiv 0 \ {\rm on} \ \partial C
\]
it follows that 
\[
\nabla u^\kappa_{\alpha} \cdot \gamma \equiv 0 \ {\rm on} \ \partial C \cap \partial D.
\]
Since $C$ is arbitrarily choosen, it follows that
 $\nabla u^\kappa_{\alpha} \cdot \gamma \equiv 0 \ {\rm on}  \ \partial D $ a.e. 

This proves that (\ref{discount_hjb}) has a solution  $ u^{\kappa}_{\alpha} \in W^{1,2, p}_{loc}((\kappa,1)\times \overline{D} ) \cap C^{\hat{\alpha}/2, \hat{\alpha}}
((0, \, 1) \times \overline{D} ), \ p\geq 2$. 
\paragraph{} Following the arguments in [\cite{menaldi_robin}, Proposition 3.2], extend the function $u^{\kappa}_{\alpha}$ 
to whole of $(0,1)$ as follows:
\begin{equation*}
\bar u^{\kappa}_{\alpha}(\theta,x)=\left\{ \begin{array}{ll}
u^{\kappa}_{\alpha}(\theta,x) & \mbox{if $\theta > \kappa$ } \\  
  \displaystyle{ e^{\frac{\kappa \|r\|_{\infty}}{\alpha} } } & \mbox{if $ 0\leq \theta \leq \kappa$ }.
\end{array}\right.
\end{equation*}
Then it follows that, $\bar u^{\kappa}_{\alpha}$ is nonnegative, bounded, continuous,
$$
\sup_{0<\kappa<1} \Big\| \frac{\partial \bar u^{\kappa}_{\alpha}}{\partial \theta}  
\Big\|_{\infty ; (0,1)\times  \overline{D}} < \infty.
$$
and for each compact $\overline{C} \subset \overline{D}$, 
$$
\sup_{0<\kappa<1} \| \bar u^{\kappa}_{\alpha} \|_{2,p ; \overline{C}} < \infty,
$$
 for each $0< \theta < 1$.
The function $\bar u^{\kappa}_{\alpha}$ is a solution in almost everywhere sense to the following p.d.e
\begin{eqnarray}\label{extended_solution}
\left.
\begin{array}{rcl}
\displaystyle{
 \alpha\theta\frac{\partial \bar u^{\kappa}_{\alpha}}{\partial \theta} } &=& \displaystyle{ \inf_{v} \left[ b(x,v)\cdot 
\nabla \bar u^{\kappa}_{\alpha} + \theta r(x,v)\bar u^{\kappa}_{\alpha}\right] +\frac{1}{2}\mbox{trace}
(a(x)\nabla^2 \bar u^{\kappa}_{\alpha}) } \\
&&- \displaystyle{ \theta e^{\frac{\kappa \|r\|_{\infty}}{\alpha}}\inf_{v\in V} \{ r(x,v) \} I_{\{\theta\leq \kappa\}} }\\
\displaystyle{ \bar u^{\kappa}_{\alpha}(0,x) } &=& \displaystyle{ 1 , \ 
\ \nabla \bar u^{\kappa}_{\alpha}.\gamma=0 \mbox{ on } \partial  D. }
\end{array}
\right\}
\end{eqnarray}
Hence $\bar u^{\kappa}_{\alpha} \in W^{1,2,p}_{loc}((0,1)\times \overline{D}  ) $  is a weak solution to 
(\ref{extended_solution}).
So multiply equation (\ref{extended_solution}) with a test function $\hat \phi \in 
C^{\infty}_c((0,1)\times \overline{D})$ 
and integrate over $(0,1)\times \overline{D}$ we get
\begin{eqnarray}\label{test}
-\alpha \int_0^1 \theta \left\langle \frac{ \partial \bar u^{\kappa}_{\alpha} }{\partial \theta}, \hat \phi\right\rangle d\theta  +
 \int_0^1 \left\langle \inf_{ v\in U} \{b(x,v) \cdot \nabla \bar u^{\kappa}_{\alpha}+ 
 \theta r(x,v)\bar u^{\kappa}_{\alpha} \}, \hat \phi \right\rangle d\theta  \nonumber \\
  + \frac{1}{2} \int_0^1 \langle \mbox{trace}(a(x) \nabla^2 \bar u^{\kappa}_{\alpha}), \hat \phi \rangle d\theta  = 
\int_0^{\kappa} \left\langle \inf_{v\in U} \{ \theta r(x,v) e^{\frac{\kappa \|r\|_{\infty}}{\alpha}} \}, \hat \phi
\right\rangle d\theta, \nonumber \\
\end{eqnarray}
where $\langle \cdot, \cdot\rangle$ is inner product on $L^2(\overline{D})$. 
By letting $\kappa \longrightarrow 0$ in above, we obtain
\begin{eqnarray*}
-\alpha \int_0^1 \theta \left\langle \frac{ \partial  u_{\alpha} }{\partial \theta}, \hat \phi\right\rangle d\theta  &+&
 \int_0^1 \left\langle \inf_{ v\in U} \{b(x,v) \cdot \nabla  u_{\alpha}+ 
 \theta r(x,v) u_{\alpha} \}, \hat \phi \right\rangle d\theta \\
 & + & \frac{1}{2} \int_0^1 \langle \mbox{trace}(a(x) \nabla^2  u_{\alpha}), \hat \phi \rangle d\theta  = 
0,
\end{eqnarray*}
where $ u_{\alpha} \in W^{1,2,p}_{loc}((0,1)\times  D ), p\geq 2$. Therefore we have
\begin{equation*}
\begin{array}{rcl}
\displaystyle{
 \alpha\theta\frac{\partial u_{\alpha}}{\partial \theta} } &=& \displaystyle{ \inf_{v \in U} 
\left[ b(x,v)\cdot 
\nabla u_{\alpha} + \theta r(x,v)u_{\alpha}\right] 
+ \frac{1}{2}trace (a(x)\nabla^2 u_{\alpha})  } \\
\displaystyle{ u_{\alpha}(0,x) } &=& \displaystyle{ 1 \mbox{ on } \overline {D}. }
\end{array}
\end{equation*}
Let $\theta \in (0,1)$ and $B$ be an open and bounded subset of $D$ with Lipschitz boundary such that its closure in $\overline{D}$ contains the part of the boundary of $D$. Clearly $\bar u^{\kappa}_{\alpha}(\theta,\cdot)$ and $u_{\alpha}(\theta,\cdot) \in  W^{2,p}( B)$. By Morrey Lemma, see [\cite{leoni}, pp. 335-339], we get $ W^{2,p}(B)$ is compactly contained in $C^{1,\hat \alpha}(\overline B)$. Hence for each fixed $\theta>0$, we have
\begin{equation*}
\bar u^{\kappa}_{\alpha}(\theta,\cdot) \longrightarrow u_{\alpha}(\theta,\cdot) \ \ \mbox{ in } C^{1,\hat \alpha}(\overline B).
\end{equation*}
Which implies that $\triangledown u_{\alpha} \cdot \gamma=0$ since $\displaystyle{  \nabla \bar u^{\kappa}_{\alpha} \cdot \gamma=0 \ \mbox{on} \ \partial D.}$ 

 Hence we have the existence of a weak solution
$ u_{\alpha} \in W^{1,2,p}_{loc}((0,1)\times  \overline{D}  ), p\geq 2 $ 
for the equation (\ref{discount_hjb}). 
This completes the proof. 
\end{proof} 

Now we prove the existence of optimal  control for the discounted risk-sensitive control problem.
From \cite{Benes}, existence of a measurable minimizing selector in (\ref{discount_hjb})
follows.
\begin{theorem}\label{Thm_discounted_rep} 
 Assume (A1) and (A2). Then equation (\ref{discount_hjb}) has a unique solution 
 $u_{\alpha} \in  W^{1,2,p}_{loc}((0, \, 1) \times \overline{D} ), \, p \geq 2$, 
given by
 \[
u_{\alpha}(\theta, x) \ = \ \inf_{v(\cdot) \in {\mathcal A}} E^v_x \Big[ 
e^{\theta \int^{\infty}_0 e^{- \alpha s}r(X_s, v_s) ds } \Big] .
\]
 Moreover if $v_{\alpha}(\cdot)$ is a minimizing selector in (\ref{discount_hjb}), then $v_{\alpha}(\cdot)$ is optimal for the $\alpha$-discounted 
risk-sensitive  control problem. 
\end{theorem}
\begin{proof} From the proof of Theorem \ref{Thm_discount_hjb} it is clear that for fixed $\theta>0$, 
$\bar u^{\kappa}_{\alpha}(\theta,x)= u^{\kappa}_{\alpha}(\theta,x)$ for sufficiently small $\kappa$. 
Mimicking the arguments used to prove (\ref{representation1}), we have the following stochastic representation 
\begin{equation*}
 u^{\kappa}_{\alpha}(\theta, x) = \inf_{v(\cdot)} E^v_x\Big[ e^{\frac{\kappa \|r\|_{\infty}}{\alpha}}
e^{\int_{0}^{T_{\kappa}} \theta e^{-\alpha s} r(X_s,v_s) ds} \Big] \ ,
\end{equation*}
where $X(\cdot)$ is the process (\ref{state}) corresponding to an admissible control $v(\cdot)$. 
Since $u^{\kappa}_{\alpha}(\theta, x) \longrightarrow u_{\alpha}(\theta, x)$ pointwise and 
$T_{\kappa} \to \infty$ as $\kappa\longrightarrow 0$ along a subsequence, using 
dominated convergence theorem, we get
$$
u_{\alpha}(\theta, x) \leq E^v_x \Big[ e^{\int_{0}^{\infty} \theta e^{-\alpha s} r(X_s,v_s) ds}\Big].
$$
Since $v(\cdot)$ is an arbitrary admissible control, we have 
$$
u_{\alpha}(\theta, x) \leq \ \inf_{v(\cdot) }
E^v_x\Big[ e^{\int_{0}^{\infty} \theta e^{-\alpha s} r(X_s,v_s) ds} \Big].
$$
In particular we get
$$
u_{\alpha}(\theta, x) \leq E^v_x \Big[ e^{\int_{0}^{\infty} \theta e^{-\alpha s} r(X_s,v_{\alpha}(X_s)) ds} \Big],
$$
where $v_{\alpha}(\cdot)$ is a minimizing selector in (\ref{discount_hjb}).
To prove other way inequality we argue as follows.
The non-negativity of the function $r$ implies $u^{\kappa}_{\alpha}(\theta, x)\geq 1$ and hence $u_{\alpha}(\theta, x)\geq 1$.
Consider the following s.d.e. 
\begin{eqnarray}\label{v_alpha}
\left.
\begin{array}{rcl}
\displaystyle{
 dX(t)} &=& \displaystyle{ b(X_t,v_{\alpha}(X_t))dt+\sigma(X_t) dW(t) - \gamma(X_t)d\xi(t) } \\
\displaystyle{ d\xi(t) }&=&  \displaystyle{  I_{\{X_t\in \partial D\}}d\xi(t)  } \\
\displaystyle{ \xi(0) } &=& \displaystyle{  0,\ \ \ \ X(0)= x\in \overline {D}. } 
\end{array}
\right\}
\end{eqnarray}
Define a sequence of stopping times as follows: 
\begin{equation*}
\tau_k=\left\{ \begin{array}{ll}
0 \ ;& \mbox{ $|x|\geq k$, } \\  
 \displaystyle{ \inf\{ t\geq 0 :  |X(t)| \geq k \} \ ; } & \mbox{ $ |x|< k$, }
\end{array}\right.
\end{equation*}
where $X(\cdot)$ is the process given by (\ref{v_alpha}) and we use the convention that infimum of an 
empty set is $+\infty$.
The resulting sequence is nondecreasing with $ \displaystyle{ \lim_{k \rightarrow \infty} \tau_k =\infty }$, a.s.
Apply Ito-Dynkin formula to $e^{\int_{0}^{t} \theta_s r(X_s,v_{\alpha}(X_s) ds}u_{\alpha}(\theta_t, X_t)$, 
we get
\begin{eqnarray*}
&& e^{\int_{0}^{T\wedge \tau_k} \theta_s r(X_s,v_{\alpha}(X_s)) ds} u_{\alpha}(\theta_{T\wedge \tau_k} , X_{T\wedge \tau_k}) \\
 &=& u_{\alpha}(\theta , x) +\int_{0}^{T \wedge \tau_k } e^{\int_{0}^{t} \theta_s r(X_s,v_{\alpha}(X_s)) ds} du_{\alpha}(\theta_t , X_t) \\ 
&&+ \int_{0}^{T\wedge \tau_k} u_{\alpha}(\theta_t , X_t) e^{\int_{0}^{t} \theta_s r(X_s,v_{\alpha}(X_s)) ds}  
\theta_t r(X_t,v_{\alpha}(t)) dt,
\end{eqnarray*}
where
\begin{eqnarray*}
du_{\alpha}(\theta_t , X_t) &=& (\nabla u_{\alpha}(\theta_t, X_t))^\bot 
\sigma(X_t)  I\{ X_t \in \partial D \} dW(t) 
- \alpha \theta_t \frac{\partial}{\partial \theta} u_{\alpha}(\theta_t , X_t) dt \\ 
&&+ \Big[ \nabla  u_{\alpha}\left(\theta_t, X_t\right) \cdot b(X_t,v_{\alpha}(X_t)) + \frac{1}{2} \mbox{trace} (a(X_t) 
\nabla^2 u_{\alpha}(\theta_t, X_t) ) \Big] dt  \\
&&- \left[\gamma(X_t)\cdot \nabla u_{\alpha}
(\theta_t, X_t)\right] I_{\{X_t\in\partial D\}} d\xi_t.
\end{eqnarray*}
Using the fact that $u_{\alpha}$ satisfy the equation (\ref{discount_hjb}), we get
\begin{eqnarray*}
&& e^{\int_{0}^{T\wedge \tau_k} \theta_s r(X_s,v_{\alpha}(X_s)) ds} u_{\alpha}(\theta_{T\wedge \tau_k} , X_{T\wedge \tau_k}) \\
&=& u_{\alpha}(\theta , x) 
 +\int_{0}^{T\wedge \tau_k} e^{\int_{0}^{t} \theta_s r(X_s,v_{\alpha}(X_s)) ds}(\nabla u_{\alpha}(\theta_t, X_t))^\bot 
\sigma(X_t) dW(t).
\end{eqnarray*}
Since $\nabla u_{\alpha}$ is continuous on $\overline{{B_k} \cap D}$  by the Sobolev embedding Theorem, therefore  $\nabla u_{\alpha}$ is bounded on $\overline{{B_k} \cap D}$, which implies that the stochastic integral 
$$
 \displaystyle{ \int_{0}^{T\wedge \tau_k} e^{\int_{0}^{t} \theta_s r(X_s,v_{\alpha}(X_s)) ds}(\nabla u_{\alpha}(\theta_t, X_t))^\bot 
\sigma(X_t) dW(t) } $$ is a zero mean martingale for each $k$.
Hence we get
\begin{equation*}
 u_{\alpha}(\theta, x) = E^v_x\left[ e^{\int_{0}^{T\wedge \tau_k}
 \theta_s r(X_s,v_{\alpha}(X_s)) ds} u_{\alpha}(\theta_{T\wedge \tau_k} , X_{T\wedge \tau_k})\right].
\end{equation*}
Letting $k\rightarrow \infty$, we get 
\begin{equation*}
 u_{\alpha}(\theta, x) = E^v_x\left[ e^{\int_{0}^{T} \theta_s r(X_s,v_{\alpha}(X_s)) ds} 
 u_{\alpha}(\theta_{T} , X_T)\right]\geq E^v_x\left[ e^{\int_{0}^{T} \theta_s r(X_s,v_{\alpha}(X_s)) ds} \right].
\end{equation*}
Now  taking $T\rightarrow \infty$, we obtain 
\begin{equation*}
 u_{\alpha}(\theta, x) \geq E^v_x\Big[ e^{ \int_{0}^{\infty} \theta e^{-\alpha s} r(X_s,v_{\alpha}(X_s)) ds } \Big] \ .
\end{equation*}
Thus,
\[
u_{\alpha}(\theta, x) \ = \ \inf_{v(\cdot) \in {\mathcal A}} E^v_x \Big[ 
e^{\theta \int^{\infty}_0 e^{- \alpha s}r(X_s, v_s) ds } \Big]
\ = \  E^v_x \Big[ e^{\theta \int^{\infty}_0 e^{- \alpha s}r(X_s, v_{\alpha}(X_s)) ds} 
 \Big] \, .
\]
This proves $v_{\alpha}(\cdot)$ is
optimal and $u_{\alpha}$ is the unique solution to the equation (\ref{discount_hjb}), which completes the proof. 
\end{proof}

\section{Risk-sensitive Control with Near Monotone Cost}
\thispagestyle{empty}
In this section we prove existence of optimal control for the risk-sensitive control problem 
described in Section \ref{PD}, under a condition on the cost function $r(x,v)$, called
``near monotonicity". We also use an additional assumption that the process given by (\ref{state})
is recurrent for each admissible control.
If $C\subset \overline{D}$, we denote by $\tau(C)$ the first exit time of the 
process $X(\cdot)$ from $C$,
$$
\tau(C)=\inf \{ t>0: X(t) \notin C\}.
$$
\begin{definitions}
Let $X(\cdot)$ be the process given by (\ref{state}) corresponding to an admissible  control $v(\cdot)$ with initial condition $x$. We say controlled process $X(\cdot)$  is {\it  recurrent}, if for any open connected set $O\subset \overline{D}$ the first
hitting time of the set $O$, i.e., $\tau(O^c)$, satisfies
 $P ( \tau(O^c)]<\infty ) = 1$, for all $x\in\overline {D}$. If 
$E [\tau(O^c)] < \infty$ for all $x \in \overline{D}$, then $X$ is said to be positive recurrent. 
Correspondingly, the control $v(\cdot)$ is called a stable  control.  
We denote the set of stable, stationary Markov controls by $\mathcal M_{s}$.
\end{definitions}

We assume that for some admissible control $v(\cdot)\in \mathcal A$ and initial condition $x\in\overline{D} $,
$$
 \limsup_{T \longrightarrow \infty}
\frac{1}{T\theta}\ln E^v_x\left[ e^{\theta\int_{0}^T r(X_t,v_t) dt}\right]  < \infty,
 $$
 where $X(\cdot)$ is the process (\ref{state}) corresponding to $v(\cdot)$. Define the optimal  risk-sensitive values as follows
\begin{eqnarray*}
\beta & = & \inf_{v(\cdot)\in\mathcal A } \limsup_{T \longrightarrow \infty}
\frac{1}{T\theta}\ln E^v_x\left[ e^{\theta\int_{0}^T r(X_t,v_t) dt}\right] .
\end{eqnarray*}

Now we state the near-monotonicity assumption. \\
{\bf (A3)}\ The cost function $r$ satisfy the following
\begin{equation}\label{near-monotone}
\liminf_{|x|\longrightarrow \infty} \inf_{v\in\mathcal V} r(x,v) > \beta, \ x\in \overline {D}, 
\end{equation}
i.e., $r$ is near monotone with respect to $\beta$. 

Also we use the following recurrent condition. 

\noindent {\bf (A4)} For each stationary Markov control $v(\cdot)$, the corresponding the process $X(\cdot)$ 
given by (\ref{state}) is recurrent.   

See Lemma 1.3 for  a characterization of (A4).

\begin{remark}
(i) Note that if $r$ is bounded then $\beta \leq \|r\|_{\infty}$.\\
(ii)It may seem at first that (\ref{near-monotone}) cannot be verified unless $\beta$ is known. However there are two important 
cases where (\ref{near-monotone}) always holds. The first is the case where $ \displaystyle{ \inf_{v\in \mathcal A 
 } 
r(x,v) } $ grows
asymptotically unbounded in $x$, and $\beta<\infty$. The second covers problems in which $r(x,v)=r(x)$ does not depend on
$v$ and $\displaystyle{ r(x) < \lim_{ |y| \rightarrow \infty } r(y) } $ for all $x\in \overline{D}$.

\end{remark}

We adapt the vanishing discount approach to prove the existence of optimal risk-sensitive ergodic control 
under the near-monotonicity assumption. To prove existence of solution for risk-sensitive ergodic HJB, we study the limiting behaviour of the equation (\ref{discount_hjb})
 as $\alpha \longrightarrow 0$.

\begin{theorem}\label{rho < beta}
Assume (A1) and (A2).  Then there exist a solution 
$(\rho, \hat u) \in \mathbb{R} \times W^{2,p}_{loc} (\overline{D })$  to the equation  
\begin{equation}
\left.
\begin{array}{rcl}\label{ergodic}
 \theta \rho \hat u &=& \displaystyle{ \inf_{v \in U} \left[ b(x,v)\cdot  \nabla \hat u + 
\theta r(x,v)\hat u \right] +\frac{1}{2}\, trace (a(x)\nabla^2 \hat u) } \\
\displaystyle{ \nabla \hat u. \, \gamma } &=& 0 \mbox{ on } \partial D, \ 
 \hat u(x_0)=1 \, .
\end{array}
\right\}
\end{equation}
Moreover $$
\rho\leq \beta.
$$
\end{theorem}
\begin{proof} Let $\chi_k$ denote a nonnegative smooth function such that $\chi_k \equiv 1 $ in $B_k$,
$\chi_k \equiv 0$ in $B^c_{k+1}$ and $0\leq \chi_k \leq 1$. Let $r_k=r\chi_k$. Then 
$$
\|r_k\|_{\infty} \leq \| r\|_{\infty, B_{k+1}}.
$$ 
Define for $\alpha >0$ 
\begin{equation}\label{value_k_alpha}
u^k_{\alpha}(\theta,x):=\inf_{v(\cdot)\in \mathcal M_{sr} } E^v_x\left[ e^{\theta\int_{0}^{\infty} e^{-\alpha t} r_k(X_t,v_t) dt} \Big| X_0 =x \right].  
\end{equation}
Consider the p.d.e.
\begin{equation}
\left.
\begin{array}{rcl}\label{monotone_k_alpha}
 \alpha \theta \frac{\partial u^k_{\alpha} }{\partial \theta} &=& \displaystyle{ \inf_{v \in U} \left[ b(x,v)\cdot  \nabla u^k_{\alpha} + \theta
 r_k(x,v) u^k_{\alpha} \right] +\frac{1}{2}\, trace (a(x)\nabla^2 u^k_{\alpha}) } \\
\displaystyle{ \nabla  u^k_{\alpha}\cdot \gamma(x) } &=& 0 \mbox{ on } \partial D, \ 
 u^k_{\alpha}(0, x)=1 \, .
\end{array}
\right\}
\end{equation}   
Mimicking the arguments as in Theorem \ref{Thm_discount_hjb} and \ref{Thm_discounted_rep}, one can see that p.d.e. (\ref{monotone_k_alpha}) has a solution $u^k_{\alpha}$ in $W^{1,2,p}_{loc}((0,1)\times D)$ with $p\geq 2$, and $u^k_{\alpha}$ has the representation (\ref{value_k_alpha}). 

Set 
\begin{equation}\label{monotone_transformation}
\phi^k_{\alpha}(\theta,x)= \frac{1}{\theta} \ln u^k_{\alpha}(\theta,x), \ g^k_{\alpha}(\theta,x)=
 \alpha \phi^k_{\alpha}  + \alpha\theta\frac{\partial \phi^k_{\alpha}}{\partial \theta} 
\end{equation}
Mimicking the arguments as in [\cite{anup_borkar_suresh},  Lemma 2.1] we have 
\begin{equation}\label{monotone_alpha1}
\| \alpha \phi^k_{\alpha} \|_{\infty} + \left\|\alpha\theta\frac{\partial \phi^k_{\alpha}}{\partial \theta}\right\|_{\infty}
\leq 3 \|r_k\|_{\infty},  \ \forall\; 0<\alpha<1,\ 0<\theta \leq 1.
\end{equation}

Let $\tau$ denote the entrance time of the process (\ref{state}) to the set $B_{k+1}$ under  
the  admissible control
 $v(\cdot)\in \mathcal A$. Let $x\in B^c_{k+1}$. Then dynamic programming principle (\ref{DPP}) gives 
\begin{eqnarray*}
u^k_{\alpha}(\theta,x) &=&  \inf_{v(\cdot)} E^v_x\left[ e^{\theta\int_{0}^{\tau^{v(\cdot)}} e^{-\alpha t} \; r_k(X_t,v_t) dt}
u^k_{\alpha}\left(\theta e^{-\alpha \tau}, X(\tau^{v(\cdot)})\right) \right] \\
&=&  \inf_{v(\cdot)} E^v_x\left[ u^k_{\alpha}(\theta e^{-\alpha \tau},X_{\tau}) 
\right] \ \ (\because r_k\equiv 0 \mbox{ on } B^c_{k+1}) \\
&\leq& \inf_{v(\cdot)} E^v_x\left[ u^k_{\alpha}(\theta ,X_{\tau}) 
\right] \ \ ( \because e^{-\alpha \tau} < 1  \mbox{ a.s. and } 
u^k_{\alpha} \mbox{ is increasing in } \theta ) \\
&\leq& \sup_{y\in \partial B_{k+1}} u^k_{\alpha}(\theta,y) 
\end{eqnarray*}
Using Harnack's inequality, from [\cite{gilbarg_trudinger}, Theorem 8.20, pp. 199] 
  we have 
\begin{equation}\label{Harnack_extension}
\sup_{x\in B_{k+1}} u^k_{\alpha}(\theta,x) 
 \leq K_{4 \cdot 1}(k),
\end{equation}
where $K_{4 \cdot 1}(k)$ is independent of $\alpha$. 

Set
$$
\bar u^k_{\alpha}(\theta,x):=\frac{ u^k_{\alpha}(\theta,x)}{u^k_{\alpha}(\theta,x_0)} 
\ \mbox{ for some } x_0\in D.
$$
Then $\bar u^k_{\alpha}$ is solution to 
\begin{equation}
\left.
\begin{array}{rcl}\label{monotone}
0 &=& \displaystyle{ \inf_{v \in U} \left[ \mathcal L \bar u^k_{\alpha}(\theta,x,v) + \theta
 (r_k(x,v)-g^k_{\alpha}) \bar u^k_{\alpha} \right] } \\
\displaystyle{ \bar \nabla  u^k_{\alpha}\cdot \gamma(x) } &=& 0 \mbox{ on } \partial D, \ 
 \bar u^k_{\alpha}(\theta,x_0)=1 \, .
\end{array}
\right.
\end{equation}   
From (\ref{Harnack_extension}) it follows that
$$
\sup_{x\in B_{k+1}} \bar u^k_{\alpha}(\theta,x) \leq K_{4 \cdot 1}(k).
$$
But the foregoing arguments show that for $x\in B^c_{k+1}$,
\begin{eqnarray*}
 \bar u^k_{\alpha}(\theta,x) &\leq& \sup_{ y\in \partial B_{k+1} } 
\left[ \frac{u^k_{\alpha}(\theta,y)}{u^k_{\alpha}(\theta,x_0)}\right] \leq K_{4 \cdot 1}(k), 
\end{eqnarray*}
where $K_{4 \cdot 1}(k)$ can be chosen independent of $x, \alpha$. Now using [\cite{gilbarg_trudinger}, Theorem 9.11, pp. 235] we have for each $R<k+1$
\begin{equation}\label{wpbound}
\|\bar u^k_{\alpha}(\theta,\cdot)\|_{W^{2,p}(B_R)} \leq K_{4 \cdot 2}, 
\end{equation}
where $K_{4 \cdot 2}>0$ is independent of $\alpha >0$. Now
using compact and continuous Sobolev embedding theorem, for each fixed $\theta>0$, without loss of generality $\theta=1$, there exists $\hat u^k \in  W^{2,p}_{loc}(D)$  such that 
\begin{eqnarray*}
\bar u^k_{\alpha}(1,\cdot) &\longrightarrow& \hat u^k  \mbox{ strongly in }   W^{1,p}_{loc}
(D), \\
\bar u^k_{\alpha}(1,\cdot) &\longrightarrow& \hat u^k  \mbox{ weakly in }  W^{2,p}_{loc}(D), 
\end{eqnarray*}
along a subsequence as $\alpha \downarrow 0$. By Sobolev embedding theorem, the convergence is uniform on compact subsets of 
$\overline{D}$, hence we have $\hat u^k$ is bounded above by $K_{4 \cdot 1}(k)$.
Now we show that 
$$
g^k_{\alpha}(1,x)  \longrightarrow \rho_k \in \mathbb R.
$$
From (\ref{monotone_alpha1}), along a further subsequence,
\begin{equation}\label{rho1}
 \alpha \phi^k_{\alpha}(\theta,x) \longrightarrow \rho^k_1(\theta,x), \ \mbox{ in weak* topology of } L^{\infty}( (0,1)\times D ).
\end{equation}
We show that $\rho^k_1$ is a function of $\theta$ alone. From (\ref{wpbound})  and $$\displaystyle{ \frac{1}{\theta} \nabla \ln \bar u^k_{\alpha}(\theta,x)= \nabla \phi^k_{\alpha}(\theta,x) ,}$$ we have for any $R>0$ 
\begin{equation}\label{monotone_interior}
\|\nabla \phi^k_{\alpha}(\theta,\cdot)\|_{W^{1,p}(B_R)} \leq K_{4 \cdot 3}, 
\end{equation}
where $K_{4 \cdot 3}>0$ is independent of $\alpha >0$. By  (\ref{monotone_interior}),
$$
\lim_{\alpha \downarrow 0} \int_{D} \alpha \nabla \phi^k_{\alpha}(\theta,x) f(x)=0,
$$
for each $f\in C^{\infty}_c(D)$. 
Thus the distributional derivative of $\rho^k_1$ in $x$ is identically zero, proving the claim. Also by  
(\ref{monotone_alpha1}), for each fixed $\theta=\theta_0>0$, $\displaystyle{ \left\{
 \alpha \frac{\partial \phi_{\alpha}^k}{\partial \theta} \Big| \alpha>0 \right\} }$ is bounded in $ L^{\infty}( [\theta_0,1]
\times D)$. Hence along a further subsequence 
\begin{equation}\label{rho2}
 \alpha \frac{\partial \phi^k_{\alpha}}{\partial \theta} \longrightarrow \rho_2^k(\theta,x), 
\ \ \mbox{ weakly in } L^{2}_{ loc }( [\theta_0,1]\times D).
\end{equation}
It follows from (\ref{rho1}) and (\ref{rho2}) that $\rho_2^k(\cdot,\cdot)=(\rho^k_1)'$ in the sense of distribution, where $(\rho^k_1)'$
is the distributional derivative (in $\theta$) of $\rho_1^k$. Hence $\rho_2^k(\cdot,\cdot)$ is also a function of $\theta$ alone.
Thus we have: for each $\theta>0$ there exists a constant $\rho_k$ such that along a subsequence
$$
\alpha \phi_{\alpha}^k+ \alpha \theta\frac{\partial \phi_{\alpha}^k}{\partial \theta} \longrightarrow \rho_k.
$$
Now letting $\alpha \longrightarrow 0$ in (\ref{monotone}) along the subsequence, we have $(\hat{u}^k, \rho_k) \in  W^{2,p}_{loc}(D) \times \mathbb{R}$ satisfying the 
following equation 
\begin{equation}
\left.
\begin{array}{rcl}\label{monotone_ergodic}
 \rho_k \hat u_k &=& \displaystyle{ \inf_{v \in U} \left[ b(x,v)\cdot  \nabla \hat u_k + 
 r_k(x,v)\hat u_k \right] +\frac{1}{2}\, trace (a(x)\nabla^2 \hat u_k) } \\
\displaystyle{ \nabla \hat u_k\cdot \gamma(x) } &=& 0 \mbox{ on } \partial D, \ 
 \hat u_k(x_0)=1 \, .
\end{array}
\right\}
\end{equation}   
Applying Ito's formula to the process (\ref{state}) corresponding to $v(\cdot)\in \mathcal A$,
\begin{eqnarray*} 
 d\left(e^{\int_{0}^t (r_k(X_s,v_s)-\rho_k) ds} \ \hat u^k(X_t)\right) &=& e^{\int_{0}^t (r_k(X_s,v_s)-\rho_k) ds} \ 
d\Big( \hat u^k(X_t)
\Big) \\ 
&+& (r_k(X_t,v_t)-\rho_k) e^{\int_{0}^t (r_k(X_s,v_s)-\rho_k) ds} \ \hat u^k(X_t)\ dt,
\end{eqnarray*}
where
\begin{eqnarray*}
d\left( \hat u^k(X_t) \right) &=&  \left[ b(X_t, v_t))\cdot \nabla \hat u^k(X_t) 
+  \frac{1}{2}\mbox{trace}(a(X_t) \nabla^2 \hat u^k(X_t)) \right] I \{ X_t \in \partial D \} dt \\ &&- 
\left(\gamma(X_t)\cdot \nabla \hat u^k(X_t)\right) I_{\{X_t\in\partial D\}}  d\xi_t
+  (\nabla \hat u^k(X_t))\sigma(X_t) dW_t \ . 
\end{eqnarray*}
Hence it follows that
\begin{eqnarray}\label{monotone_binq}
e^{\int_{0}^{T\wedge\tau_R} (r_k(X_s,v_s)-\rho_k) ds} \hat u^k(X_{T\wedge\tau_R})&-&\hat u^k(x) \nonumber\\
&\geq&  
\int_{0}^{T\wedge\tau_R} e^{\int_{0}^t (r_k(X_s,v_s)-\rho_k) ds} \nabla \hat u^k(X_t)\sigma(X_t) dW_t. 
\end{eqnarray}
Since we have $ \hat u^k \in  W^{2,p}_{loc}(D), p \geq d $, 
which implies $\nabla \hat u^k$ is bounded on each compact subset $\overline{B}_R$, hence
$$
\int_{0}^{T\wedge\tau_R} e^{\int_{0}^t (r_k(X_s,v_s)-\rho_k) ds} (\nabla \hat u^k(X_t))\sigma(X_t) dW_t,
$$
is a zero mean martingale. Taking expectation in (\ref{monotone_binq}) we obtain
\begin{eqnarray*}
E^v_x\left[ e^{\int_{0}^{T\wedge\tau_R} (r_k(X_s,v_s)-\rho_k) ds} \hat u^k(X_{T\wedge\tau_R})\right] -\hat u^k(x) \geq 0.
\end{eqnarray*}
Since $\hat u^k$ is bounded above, we have
$$
\hat u^k(x) \leq K_{4 \cdot 1}(k) E^v_x\left[ e^{\int_{0}^{T\wedge\tau_R} (r_k(X_s,v_s)-\rho_k) ds}\right]
\leq K_{4 \cdot 1}(k) E^v_x\left[ e^{\int_{0}^{T} (r_k(X_s,v_s)-\rho_k) ds}\right] .
$$
Taking $\ln$ and divide by $T$ we get
$$
\frac{1}{T} \ln \hat u^k(x) \leq \frac{\ln K(k)}{T}+ \frac{1}{T}\ln E^v_x\left[ e^{\int_{0}^{T} (r_k(X_s,v_s)-\rho_k) ds}\right].
$$
Since $u^k_{\alpha}\geq 1$ by definition $\bar u^k_{\alpha}$ is bounded below, 
hence uniform convergence on compact sets gives that $\hat u^k$ is bounded below say by $K_{4 \cdot 3}(k)>0$ . 
Hence 
$$
\frac{1}{T} \ln K_{4 \cdot 3}(k) \leq \frac{\ln K(k)}{T}+ \frac{1}{T}\ln E^v_x\left[ e^{\int_{0}^{T} r_k(X_s,v_s ds} \right]  -\rho_k.
$$
Now taking $T \longrightarrow\infty$ we get
$$
\rho_k \leq \limsup_{T\longrightarrow\infty} \frac{1}{T}\ln E^v_x\left[ e^{\int_{0}^{T} r_k(X_s,v_s) ds}\right].
$$
Since $|r_k| \leq |r|$, 
$$
\rho_k \leq \limsup_{T\longrightarrow\infty} \frac{1}{T}\log E^v_x\left[ e^{\int_{0}^{T} r(X_s,v_s) ds}\right].
$$
Taking infimum over all stable stationary Markov controls in the right hand side of above, we get
\begin{equation}\label{k_buniform}
\rho_k \leq \beta, \  \forall\; k.
\end{equation}
Since the coefficients of (\ref{monotone_ergodic}) are bounded, we have $|\hat u^k| $ is bounded uniformly
in $k$ on compact sets by Harnack's inequality. Thus  we have 
$\hat u^k \longrightarrow \hat u$ in $W^{1,p}_{loc}(D)$ and $\rho_k\longrightarrow \rho$ along
a subsequence. Furthermore, it follows from Harnack's inequality that $\hat u >0$ on compacts, in fact one has uniform positive lower bounds for $\hat u^k$ on compacts. Letting $k\longrightarrow\infty$ in (\ref{monotone_ergodic}),  $(\rho,\hat u)$ satisfy 
\begin{equation}
\left.
\begin{array}{rcl}\label{monotone_ergodicfinal}
 \rho \hat u &=& \displaystyle{ \inf_{v \in U} \left[ b(x,v)\cdot  \nabla \hat u + 
 r(x,v)\hat u \right] +\frac{1}{2}\, trace (a(x)\nabla^2 \hat u) } \\
\displaystyle{ \nabla \hat u\cdot \gamma(x) } &=& 0 \mbox{ on } \partial D, \ 
 \hat u(x_0)=1 \, ,
\end{array}
\right\}
\end{equation} 
where for the boundary condition it is same argument as in Theorem \ref{Thm_discount_hjb}.
In view of (\ref{k_buniform}) it follows that $\rho \leq \beta$, which completes the proof. \end{proof}
\begin{theorem}\label{thm_near-monotone}
Assume (A1)-(A4). Then ergodic risk-sensitive HJB equation (\ref{ergodic}) has a solution $(\rho,\hat \phi)$ such 
that $\rho$ is unique and is characterized by $\rho=\beta$. Also, minimizing selector in (\ref{ergodic}) is an optimal control. 
\end{theorem}
\begin{proof} In view of Theorem \ref{rho < beta} it remains to show $\beta \leq \rho$.
By assumption (\ref{near-monotone}) we have 
$$ 
\inf_{v} r(\cdot, v) \geq \beta > \rho \ \ \mbox{ outside } O \subset
\overline{D}. 
$$ 
We know that for some  $\nu>0$, $\hat u\geq \nu>0$ in $O$. Let $x\in O^c \cap \overline{D}$. 
Let $R>0$ be large enough so that $B_R$ contains the set $O$ and initial point $x$. Set $T_R = R \wedge 
\tau(B_R)$.
Let $v^*(\cdot)$ be  minimizing selector in (\ref{ergodic}), applying Dynkin's formula 

\begin{eqnarray*}
&& e^{\int_{0}^{\tau(O^c) \wedge  T_R} (r(X_s,v^*_s)-\rho) ds} \hat u(X_{\tau(O^c) \wedge  T_R})-\hat u(x) \\ & = & 
\int_{0}^{\tau(O^c) \wedge  T_R} e^{\int_{0}^t (r(X_s,v^*_s)-\rho) ds} (\nabla \hat u(X_t))^\bot
\sigma(X_t) dW_t.
\end{eqnarray*}

Since $\hat u \in  W^{2,p}_{loc}(\overline{D})$, 
it follows that $\nabla \hat u$ is locally bounded and using the boundedness of $r,\sigma$,  
$$
\int_{0}^{t \wedge \tau(O^c) \wedge  T_R}  e^{\int_{0}^{t'} (r(X_s,v^*_s)-\rho) ds} (\nabla \hat u(X_{t'}))\sigma(X_{t'}) dW_{t'}, $$ is zero mean martingale.
Then we have
\begin{eqnarray*}
E^v_x\left[ e^{\int_{0}^{\tau(O^c) \wedge T_R } (r(X_s,v^*_s)-\rho) ds} \hat u(X_{\tau(O^c) \wedge T_R) })\right] -\hat u(x) = 0 
\end{eqnarray*}
Using the Fatou's lemma, letting $R\longrightarrow\infty$ we get
$$
\hat u (x)\geq E^v_x \left[ e^{\int_{0}^{\tau(O^c)} (r(X_t, v^*_t) -\rho) dt} \hat u(X_{\tau(O^c)})\right]. 
$$
Using (A4), it follows that $\tau(O^c) < \infty$ a.s. Hence
$$
E^v_x \left[ e^{\int_{0}^{\tau(O^c)} (r(X_t, v^*_t) -\rho) dt} \hat u(X_{\tau(O^c)})\right] \geq \nu.
$$
This proves that $\hat u$ is bounded below by $\nu$. Repeating the previous argument, we also have for 
any $T>0$,
$$
\hat u (x)\geq E^v_x \left[ e^{\int_{0}^T (r(X_t, v^*_t) -\rho) dt} \hat u(X_T)\right] \geq \nu 
E^v_x \left[ e^{\int_{0}^T (r(X_t, v^*_t) -\rho) dt} \right] .
$$
Taking logarithm and dividing by $T$ 
$$
\frac{1}{T} \ln E^v_x \left[ e^{\int_{0}^T (r(X_t, v^*_t) -\rho) dt} \right] + \frac{1}{T} \ln \nu  \leq \frac{1}{T} \ln { \hat u (x) }
$$
Letting $T\longrightarrow\infty$ on both sides, we have 
$$
\limsup_{T \longrightarrow \infty} \frac{1}{T} \ln E^v_x \left[ e^{\int_{0}^T (r(X_t, v^*_t) -\rho) dt} \right]  \leq 0
$$
i.e.,
$$
\limsup_{T \longrightarrow \infty} \frac{1}{T} \ln E^v_x \left[ e^{\int_{0}^T r(X_t, v^*_t)  dt} \right]  \leq \rho.
$$
Thus $\beta \leq \rho$. 
This completes the proof of the theorem. \end{proof}

\addcontentsline{toc}{chapter}{Bibliography}
 
\end{document}